\documentclass[11pt]{amsart}
\usepackage{geometry}
\usepackage{amssymb, amsmath, amsthm, color, tikz,enumerate,mathrsfs}
\usepackage{pdfpages}
\usepackage{hyperref}
\usepackage{textcomp}
\usepackage{color}
\usepackage{xcolor}
\usepackage{listings}
\usepackage{todonotes}

\usepackage{fullpage}

\newtheorem{lemma}{Lemma}[section]
\newtheorem{proposition}[lemma]{Proposition}
\newtheorem{theorem}[lemma]{Theorem}
\newtheorem{corollary}[lemma]{Corollary}
\newtheorem{conjecture}[lemma]{Conjecture}
\newtheorem{example}[lemma]{Example}
\newtheorem{remark}[lemma]{Remark}

\theoremstyle{definition}
\newtheorem{definition}[lemma]{Definition}

\makeatletter
\newcommand{\setword}[2]{%
  \phantomsection
  #1\def\@currentlabel{\unexpanded{#1}}\label{#2}%
}
\makeatother

\numberwithin{equation}{section}

\newcommand{\Ball}{\mathbb{B}}
\newcommand{\Ccc}{\mathbb{C}}

\newcommand{\Rrr}{\mathbb{R}}
\newcommand{\Sss}{\mathbb{S}}

\newcommand{\Zzz}{\mathbb{Z}}
\newcommand{\R}{\Rrr}

\newcommand{\Z}{\Zzz}
\newcommand{\C}{\Ccc}
\newcommand{\bl}[2]{(#1,#2)}

\newcommand{\circdots}{\circ \cdots \circ}
\newcommand{\Hf}{\mathcal{H}}
\newcommand{\Tf}{\mathscr{T}}

\newcommand{\fl}{\longrightarrow}

\newcommand{\ungras}{1\!\!\mkern -1mu1}

\DeclareMathOperator{\im}{imult}
\DeclareMathOperator{\id}{id}
\DeclareMathOperator{\Span}{Span}
\DeclareMathOperator{\Vol}{Vol}
\DeclareMathOperator{\divv}{div}

\makeatletter
\@namedef{subjclassname@2020}{%
  \textup{2020} Mathematics Subject Classification}
\makeatother

\begin{document}

\title{Sharing pizza in n dimensions}

\author{Richard EHRENBORG, Sophie MOREL and Margaret READDY}

\address{Department of Mathematics, University of Kentucky, Lexington,
  KY 40506-0027, USA.\hfill\break \tt http://www.math.uky.edu/\~{}jrge/,
  richard.ehrenborg@uky.edu.}

\address{Department of Mathematics and Statistics, ENS de Lyon,
Unit\'e De Math\'ematiques Pures Et Appliqu\'ees,
69342 Lyon Cedex 07,
France.\hfill\break
\tt http://perso.ens-lyon.fr/sophie.morel/,
sophie.morel@ens-lyon.fr.}

\address{Department of Mathematics, University of Kentucky, Lexington,
  KY 40506-0027, USA.\hfill\break
\tt http://www.math.uky.edu/\~{}readdy/,
margaret.readdy@uky.edu.}


\date{\today.}

\begin{abstract}
We introduce and prove the $n$-dimensional Pizza Theorem: Let $\Hf$ be a hyperplane arrangement in~$\Rrr^{n}$. If $K$ is a measurable set of finite volume, the \emph{pizza quantity} of $K$ is the alternating sum of the volumes of the regions obtained by intersecting $K$ with the arrangement~$\Hf$. We prove that if $\Hf$ is a Coxeter arrangement different from $A_{1}^{n}$ such that the group of isometries $W$ generated by the reflections in the hyperplanes of $\Hf$ contains the map $-\id$, and if~$K$ is a translate of a convex body that is stable under $W$ and contains the origin, then the pizza quantity of~$K$ is equal to zero. Our main tool is an induction formula for the pizza quantity involving a subarrangement of the restricted arrangement on hyperplanes of $\Hf$ that we call the \emph{even restricted arrangement}. More generally, we prove that for a class of arrangements that we call \emph{even} (this includes the Coxeter arrangements above) and for a \emph{sufficiently symmetric} set $K$, the pizza quantity of $K+a$ is polynomial in $a$ for $a$ small enough, for example if $K$ is convex and $0\in K+a$. We get stronger results in the case of balls, more generally, convex bodies bounded by quadratic hypersurfaces. For example, we prove that the pizza quantity of the ball centered at $a$ having radius $R\geq\|a\|$ vanishes for a Coxeter arrangement $\Hf$ with $|\Hf|-n$ an even positive integer. We also prove the Pizza Theorem for the surface volume: When $\Hf$ is a Coxeter arrangement and  $|\Hf| - n$ is a nonnegative even integer, for an $n$-dimensional ball the alternating sum of the $(n-1)$-dimensional surface volumes of the regions is equal to zero.

\end{abstract}

\maketitle

\section{Introduction}

Given a disc in the plane
select any point in the disc.
Cut the disc by four lines through this point
that are equally spaced.
We obtain eight slices of the disc, each having
angle $\pi/4$ at the point.
The alternating sum of the areas of these eight slices is equal to zero.
This is known as the Pizza Theorem and was first stated
as a problem in Mathematics Magazine by Upton~\cite{Upton}
and solved by Goldberg~\cite{Goldberg}.
There are many two-dimensional extensions of this result;
see~\cite{Frederickson,Mabry_Deiermann}
and the references therein.
An especially interesting solution to the original problem
is by Carter and Wagon~\cite{Carter_Wagon},
who prove the result by a dissection.

In hyperplane arrangement terminology four equally spaced lines through one point
is equivalent to the $B_{2}$ arrangement.
Goldberg~\cite{Goldberg} 
pointed out that the Pizza Theorem also holds for $2k$ equally spaced
lines through a point inside a disc, that is,
the dihedral arrangement~$I_{2}(2k)$.
We extend these results to
hyperplane arrangements in higher dimensions.
Given a base chamber, every chamber~$T$ of an arrangement
has a natural sign $(-1)^{T}$.
We define the {\em pizza quantity} of a measurable set~$K$ with
finite volume
to be the alternating sum
$P(\Hf, K) = \sum (-1)^{T} \cdot \Vol(K \cap T)$,
where the sum ranges over all chambers~$T$ of the arrangement~$\Hf$;
see equation~\eqref{equation_pizza_definition}.

We first establish an expression that relates the variation of the pizza
quantity $P(\Hf,K)$ when we translate a set~$K$ to pizza quantities
of arrangements on hyperplanes in $\Hf$
known as the even restricted arrangements;
see Theorem~\ref{theorem_reduction}.
Using this formula and induction on the dimension of $V$, we show
that if $\Ball(a,R)$ is the closed ball with center $a$ and radius
$R$ and if $|\Hf|$ and $\dim(V)$ have the same parity, 
then the pizza quantity $P(\Hf,\Ball(a,R))$
is polynomial in
the pair $(R,a)$, homogeneous of degree $\dim(V)$ and 
only having terms of even degree in~$R$ as long as $\Ball(a,R)$ contains
the origin; see Theorem~\ref{theorem_second_in_introduction}.
When~$|\Hf|$ and $\dim(V)$ do not have have the same parity then we
can only say that $P(\Hf,\Ball(a,R))$ is a real analytic function of
$(a,R)$. If the arrangement $\Hf$ has enough symmetries, we can use
them to show that the lower degree terms of
$P(\Hf,\Ball(a,R))$ vanish, and in some cases,
that the pizza quantity~$P(\Hf,\Ball(a,R))$ itself vanishes.
More precisely, we obtain the following result
(see Theorem~\ref{theorem_almost_all_the_products} and
Corollary~\ref{corollary_going_to_infinity} for more general statements):

\begin{theorem}
Let $\Hf$ be a Coxeter arrangement 
on a finite-dimensional
inner product space $V$ such that $|\Hf| \geq \dim(V)$.
Assume that the ball $\mathbb{B}(a,R)$ contains the origin,
that is, $R \geq \|a\|$.
\begin{itemize}
\item[(i)]
If the number of hyperplanes is strictly greater than the dimension of~$V$
and has the same parity as that dimension
then the pizza quantity $P(\Hf,\mathbb{B}(a,R))$ vanishes.
\item[(ii)]
If the number of hyperplanes is strictly greater than the dimension of~$V$
and does not have the same parity as that dimension then
the pizza quantity $P(\Hf,\mathbb{B}(a,R))$ tends to $0$ as
the radius~$R$ tends to infinity (when the center~$a$ is fixed).
\item[(iii)]
If the number of hyperplanes is equal to the dimension of $V$
then
$P(\Hf,\mathbb{B}(a,R))$ is independent of the radius~$R$.
\end{itemize}
\label{theorem_second_in_introduction}
\end{theorem}

In the case when $V$ is $3$-dimensional
and
$\Hf$ is the arrangement of type~$A_{1} \times I_{2}(2k)$,
this corollary
specializes to the Calzone Theorem;
see~\cite{Berzsenyi}
and~\cite[page~32]{Frederickson}.

Note that if $V$ is $1$-dimensional, if $\Hf$ consists of the
hyperplane $\{0\}$ and if $K$ is a segment centered at $0$, then the
pizza quantity $P(\Hf,K+a)$ is equal to $2a$ as long as $0\in K+a$, so
it is polynomial in $a$ and independent of $K$. Building on this
observation, we prove a similar result in higher dimensions for
an \emph{even} arrangement $\Hf$ (see Definition~\ref{def_even})
and a measurable set that is \emph{sufficiently symmetric} with
respect to $\Hf$ (see Definition~\ref{definition_sufficiently_symmetric}).
In particular, if $\Hf$ is a Coxeter arrangement then
it is even if and only if $-\id_V$ is in its Coxeter group
(Corollary~\ref{corollary_nested_even_arrangements_Coxeter}).
Furthermore, we show any measurable set
stable by its Coxeter group is sufficiently symmetric
(Corollary~\ref{corollary_sufficiently_symmetric_Coxeter}).
We then conclude the following result, which is a particular
case of Theorem~\ref{theorem_amazing}
(recall that a convex body is a compact convex set):

\begin{theorem}
Let $\Hf$ be a Coxeter arrangement 
on a finite-dimensional
inner product space $V$ such that
the map $-\id_{V}$
belongs to the Coxeter group.
Equivalently,
let $\Hf$ be a product arrangement where the factors
are from the types
$A_{1}$,
$B_{n}$ for $n \geq 3$,
$D_{2m}$ for $m \geq 2$,
$E_{7}$, $E_{8}$, $F_{4}$,
$H_{3}$, $H_{4}$
and $I_{2}(2k)$ for $k \geq 2$.
Let $K$ be a convex body stable under
reflections in the hyperplanes of the arrangement~$\Hf$
such that the translate $K+a$ contains the origin.
Then the pizza quantity $P(\Hf,K+a)$ is given by
\begin{align*}
P(\Hf,K+a)
& =
\begin{cases}
2^{n} \cdot a_{1} \cdots a_{n} &
\text{if $\Hf$ is of type $A_{1}^{n}$,} \\
0 &
\text{otherwise.}
\end{cases}
\end{align*}
When the arrangement $\Hf$ has type $A_{1}^{n}$,
we assume that
it is given by the coordinate hyperplanes $\{x_i=0 : 1\leq i\leq n\}$,
the base chamber is $T_{0}=(\Rrr_{>0})^{n}$
and the point $a$ is given by $a = (a_{1},\ldots,a_{n})$.
\label{theorem_first_in_introduction}
\end{theorem}

The paper is organized as follows.
In Section~\ref{section_Initial_remarks}
we introduce the pizza quantity,
review Coxeter and product arrangements,
and point out some basic properties of these notions.
In Section~\ref{section_The_even_restricted_arrangement}
we define the even restricted arrangement.
This is a subarrangement of the restricted arrangement.
Using this notion we develop a recursion 
to compute
the pizza quantity; see Theorem~\ref{theorem_reduction}.
In Section~\ref{section_condition_E}
we introduce the notion of
an even restriction sequence,
that is, a sequence iterating the notion
of the even restricted arrangement.
This yields the notion of an {\em even arrangement},
that is, 
an arrangement such that
every even restriction sequence extends to a sequence
that has length equal to the dimension of the space.
In Corollary~\ref{corollary_nested_even_arrangements_Coxeter}
we classify all the even Coxeter arrangements.
One equivalent condition
is that the negative of the identity map belongs to the Coxeter group.
In Section~\ref{section_sufficiently_symmetric}
we introduce the notion of a sufficiently symmetric measurable set.
For sufficiently symmetric measurable sets $K$ of finite volume
and for $a\in V$ satisfying some conditions,
we show for example that the pizza quantity of $K+a$
only depends on the shift $a$
and is given by the associated polynomial
of the even arrangement;
see Theorem~\ref{theorem_pizza_is_polynomial}.
In Section~\ref{section_ball}
we restrict our attention to balls
and to convex bodies bounded by quadratic surfaces.
For a Coxeter arrangement~$\Hf$ such that $|\Hf|>\dim(V)$, we show
that the pizza quantity $P(\Hf,\mathbb{B}(a,R))$ is equal to~$0$
if $|\Hf|$ and $\dim(V)$ have the same
parity, and that otherwise
$P(\Hf,\mathbb{B}(a,R)) \longrightarrow 0$ as $R \longrightarrow +\infty$
(with the center $a$ fixed).
Moreover if $\Hf=\dim(V)$ we show that
$P(\Hf,\mathbb{B}(a,R))$ is independent of~$R$. In all these cases,
we always assume that the ball $\mathbb{B}(a,R)$ contains the origin.
In Section~\ref{section_surface_volume}
we look at the case of the $n$-dimensional ball
and consider the alternating sum of 
the surface volume of the regions.
This is the Pizza Theorem for the $(n-1)$st intrinsic volume;
see Theorem~\ref{theorem_surface_volume_two}.
Finally, in Section~\ref{section_concluding_remarks}
we briefly consider the problem of sharing pizza among
more than two people, and state
some concluding remarks
and open questions.

\section{Initial remarks and definitions}
\label{section_Initial_remarks}

Let $V$ be an $n$-dimensional real vector
space endowed with an inner product: for $v,w\in V$, we denote
their inner product by $\bl{v}{w}$.
Let $E$ be a collection of {\em unit} vectors in~$V$
such that no two vectors of $E$ are linearly dependent.
In other words, the intersection $E \cap (-E)$ is empty.
By restricting to unit vectors, we will avoid having
normalization factors in our expressions.
Let $\Hf$ be the central hyperplane arrangement corresponding to~$E$,
that is, $\Hf = \{H_{e} : e \in E\}$ where
$H_{e} = \{x \in V : \bl{e}{x} = 0\}$.
Note that each hyperplane contains the origin. 
Let $|\Hf|$ denote the number of hyperplanes in the arrangement~$\Hf$,
that is, the cardinality of the index set $E$.
A {\em chamber} of a hyperplane arrangement is a connected
component of the complement of the arrangement in~$V$.
Let $\Tf=\Tf(\Hf)$ be the collection of chambers of~$\Hf$.
For two chambers~$T_{1}$ and~$T_{2}$ in~$\Tf$,
define the {\em separation set} $S(T_{1},T_{2})$
to be the set of all indices $e \in E$
such that the two chambers lie on different sides of the hyperplane~$H_{e}$.

The arrangement $\Hf$ is oriented by the data of the set~$E$: if
$e\in E$ then the hyperplane $H_{e}$ cuts~$V$ into a positive and a
negative half-space, where the positive half-space is the one containing
the vector $e$. If $T$ is a chamber of $\Hf$ then it gives rise
to a sign vector in $\{\pm\}^E$ whose $e$-component, for $e\in E$, is
$+$ if and only if $T$ is included in the positive half-space bounded by
$H_{e}$. We assume that there is a chamber $T_{0}$ whose corresponding
sign vector only has $+$ components, and we call it the \emph{base chamber}.
The base chamber also determines the direction of each vector in the set~$E$:
for each $H\in\Hf$, the
corresponding vector is the unique normal unit vector $e$ of $H$ such that
$T_{0}$ and $e$ are on the same side of $H$.
Each chamber $T$ has a {\em sign}~$(-1)^{T}$,
which is $-1$ to the number of hyperplanes of the arrangement 
one must pass through
when walking from $T_{0}$ to $T$,
that is,
$(-1)^{T} = (-1)^{|S(T_{0},T)|}$.
For a reference on hyperplane arrangements,
see~\cite[Chapter~1]{OM}
or~\cite[Section~3.11]{Stanley_EC_I}.

Let $a$ be a point in the vector space $V$ 
and let $R$ be a nonnegative real number.
Let $\Ball(a,R)$ be the ball of radius $R$ centered at $a$,
that is,
\begin{align*}
\Ball(a,R) & = \{x \in V : \| x-a \| \leq R \} . 
\end{align*}

The space $V$ with its inner product is isomorphic to $\R^n$, and
we denote by $\Vol_V$ or just $\Vol$ 
the pullback by such an isomorphism of Lebesgue
measure on $\R^n$.
This does not depend on the choice of the isomorphism,
because Lebesgue measure is invariant under isometries. 
For every Lebesgue measurable subset $K$ of $V$ that has finite volume,
define the {\em pizza quantity} for $K$ to be
the alternating sum
\begin{align}
P(\Hf,K)
& = 
\sum_{T \in \Tf} (-1)^{T} \cdot \Vol(K \cap T).
\label{equation_pizza_definition}
\end{align}
There is a slight abuse of notation here:
the quantity in~\eqref{equation_pizza_definition}
not only depends on the arrangement~$\Hf$,
but also on the base chamber $T_{0}$.
More generally, if $f:V \longrightarrow \Ccc$ is any $L^1$ function,
we define the {\em pizza quantity} for $f$ to be
\begin{align*}
P(\Hf,f)
& = 
\sum_{T \in \Tf} (-1)^{T} \cdot \int_{T} f(x) \: dV .
\end{align*}
We then have $P(\Hf,K)=P(\Hf,\ungras_K)$, where $\ungras_K$
is the characteristic function of $K$.

The $2$-dimensional Pizza Theorem is as follows:
\begin{theorem}[Goldberg~\cite{Goldberg}]
Let $\Hf$ be the dihedral arrangement $I_{2}(2k)$ in $\Rrr^{2}$
for $k \geq 2$.
For every $a \in \Rrr^{2}$ and every $R \geq \|a\|$,
the pizza quantity for the ball $\Ball(a,R)$ vanishes:
\begin{align*}
P(\Hf, \Ball(a,R)) & = 0.
\end{align*}
\end{theorem}

A hyperplane arrangement $\Hf$ is a {\em Coxeter arrangement}
if the group $W$ generated by the orthogonal reflections
in the hyperplanes of $\Hf$ is finite and the
arrangement is closed under all such reflections.
There is a vast literature concerning Coxeter arrangements
and their associated root systems; see~\cite{Bourbaki,Humphreys}.
The group $W$ is known as the {\em Coxeter group} of the arrangement.

We say a subset $L$ of $V$ is {\em stable}
with respect to a Coxeter group $W$ acting on $V$
if $W$ leaves it invariant, that is,
$w(L) = L$ for all $w \in W$.
We define the translate of  
$L$ by $a\in V$ to be the set $L+a = \{x+a : x \in L\}$.
If $f:V \longrightarrow\C$ is a function, we say that it is {\em stable}
with respect to a Coxeter group $W$ if $f(w(x))=f(x)$ for
every $w\in W$ and $x\in V$.
We then have
that a subset $L$ of $V$ is stable by $W$ if and only if
$\ungras_L$ is stable.
For $a \in V$, we denote
the shift of function $x \longmapsto f(x-a)$ by $f_a$.
Then we have
$(\ungras_{L})_{a}=\ungras_{L+a}$
for every $a \in V$.

\begin{proposition}
Let $u : V \longrightarrow V'$ be an isometry, 
where $V'$ is another real inner product space, and let
$\Hf$ be a hyperplane arrangement in~$V$. 
Let $u(\Hf)$ be the hyperplane arrangement 
$\{u(H) : H \in \Hf\}$ with base chamber $u(T_{0})$.
For any $L^1$ function $f:V' \longrightarrow \C$,
the following equality holds:
\begin{align*}
P(u(\Hf), f) & = P(\Hf, f \circ u) .
\end{align*}
\label{proposition_equivariance_pizza_quantity}
\end{proposition}
\begin{proof}
The map $u$ induces a bijection between the chambers of
$\Hf$ and those of $u(\Hf)$ that respects
the cardinality of the separation set,
that is,
$|S(T_{1},T_{2})| = |S(u(T_{1}),u(T_{2}))|$.
The conclusion follows from the fact that $u$, being an isometry,
is volume-preserving.
\end{proof}

\begin{corollary}
Let $\Hf$ be a Coxeter arrangement with Coxeter group $W$
and $f:V \longrightarrow\C$ be an $L^1$~function.
Then for every $w\in W$ we have
\begin{align*}
P(\Hf,f\circ w) & = \det(w) \cdot P(\Hf,f) . 
\end{align*}
In particular, if $f$ is stable by a reflection in $W$
then we have $P(\Hf,f)=0$.
\label{corollary_Coxeter_on_a_hyperplane}
\end{corollary}
\begin{proof}
The first statement is just 
Proposition~\ref{proposition_equivariance_pizza_quantity},
because $w(\Hf)=\Hf$
and
$(-1)^{|S(T_{0},w(T_{0}))|} = \det(w)$
for every $w\in W$.
The second statement follows from the first and from
the fact that reflections have determinant $-1$.
\end{proof}

\begin{remark}
{\rm
Note that the situation
where $f$ is stable by a reflection $s$ in $W$
occurs if $f=g_a$, with $g:V \longrightarrow\C$ an $L^1$ function
that is stable by $W$ and $a$ a point belonging
to the hyperplane $H$ of the arrangement~$\Hf$,
where $s$ is the orthogonal reflection in~$H$.
}
\label{remark_stable}
\end{remark}

Given two hyperplane arrangements
$\Hf_{1} = \{H_{e}\}_{e \in E_{1}}$
and
$\Hf_{2} = \{H_{e}\}_{e \in E_{2}}$
in the vector space~$V_{1}$, respectively~$V_{2}$,
define the product arrangement $\Hf_{1} \times \Hf_{2}$
in $V_{1} \times V_{2}$, where the index set of
vectors is $E = (E_{1} \times \{0\}) \cup (\{0\} \times E_{2})$.
Note that in this construction
the hyperplanes inherited from~$\Hf_{1}$
are orthogonal to the hyperplanes inherited from~$\Hf_{2}$.
Furthermore, if $T_{i,0}$ is the base chamber of~$\Hf_{i}$
then $T_{1,0} \times T_{2,0}$ is 
the base chamber of $\Hf_{1} \times \Hf_{2}$.

\section{The even restricted arrangement}
\label{section_The_even_restricted_arrangement}

In this section we obtain a recursion for evaluating the pizza quantity
in terms of lower dimensional pizza quantities
on certain subarrangements of the given arrangement.
We begin by introducing two definitions.

\begin{definition}
\begin{enumerate}[(a)]
\item
Let $V' \subseteq V$ be a subspace of codimension $2$.
The \emph{intersection multiplicity} of the arrangement $\Hf = \{H_{e} : e \in E\}$
at the subspace $V'$ is the cardinality
\begin{align*}
\im(V') & = |\{e\in E : H_{e} \supseteq V'\}|.
\end{align*}
\item
For $e \in E$ the {\em even restricted arrangement}
$\Hf_{e}$ is the arrangement inside the vector space~$H_{e}$
consisting of the hyperplanes $H_{e} \cap H_{f}$ (these are
codimension $2$ subspaces of $V$) with
even intersection multiplicity, that is,
\begin{align*}
\Hf_{e}
& =
\{H_{e} \cap H_{f} : f \in E, e \neq f, \im(H_{e} \cap H_{f}) \equiv 0 \bmod 2\} .
\end{align*}
\end{enumerate}
\end{definition}

The even restricted arrangement
$\Hf_{e}$ is a subarrangement of the {\em arrangement~$\Hf$
restricted to the hyperplane~$H_{e}$},
for brevity known as the restricted arrangement,
that is,
$\Hf_{e}^{\prime\prime} = \{H_{e} \cap H_{f} : f \in E, e \neq f\}$;
see~\cite[Section~3.11.2]{Stanley_EC_I}.
In general these two arrangements are different.
Even though the letter $e$ appears in the notation for $\Hf_{e}$, the arrangement
$\Hf_{e}$ only depends on $H_{e}$, and not on the choice of its normal
vector $e$.

\begin{proposition}
Suppose that $\Hf$ is a Coxeter arrangement, and let $e\in E$.
\begin{itemize}
\item[(i)] Let $e'\in E - \{e\}$. Then the intersection multiplicity of $\Hf$ at
$H_e\cap H_{e'}$ is even if and only if there exists $f\in E - 
\{e\}$ such that $H_e\cap H_{f}=H_e\cap H_{e'}$ and that
$(e,f)=0$.

\item[(ii)] Let $F=\{f\in E : (e,f)=0\}=E\cap H_e$. Then
$\Hf_e$ is the hyperplane arrangement $(H_e\cap H_f)_{f\in F}$ on $H_e$, and
it is a Coxeter arrangement.

\end{itemize}
\label{prop_even_arrangement_Coxeter}
\end{proposition}

\begin{proof}
\begin{itemize}
\item[(i)]
Let $E'=\{f\in E - \{e\} : H_e\cap H_{f}=H_e\cap H_{e'}\}$.
Then $|E'|+1$ is the intersection multiplicity of $\Hf$ at
the intersection $H_e \cap H_{e'}$, and we have
$s_e(E')=E'$, where $s_e$ is the orthogonal reflection in the hyperplane $H_e$.
As $s_e^2=1$, the set $E'$ has odd cardinality if and only if
$s_e$ has a fixed point on $E'$. But $f\in E'$ is fixed by
$s_e$ if and only if $s_e(H_{f})=H_{f}$, which is equivalent
to the condition that $(e,f)=0$.

\item[(ii)] The first statement follows from~(i). It remains to prove that
$\Hf_e$ is a Coxeter arrangement. For every
$f\in F$, if $s_f$ is the orthogonal reflection in the hyperplane $H_f$,
then $s_f(H_e)=H_e$ and $s_f(F)=F$.
In particular, $s_f$ preserves the arrangement~$\Hf_e$.
We conclude that $\Hf_e$ is a Coxeter arrangement on~$H_e$.

\end{itemize}
\end{proof}

For a hyperplane $H_{e}$
define the open half spaces 
$H_{e}^{+}$ and $H_{e}^{-}$
to be
\begin{align*}
H_{e}^{+} & = \{x \in V : \bl{e}{x} > 0\}
\hspace*{-20 mm}
& \text{and} &&
\hspace*{-20 mm}
H_{e}^{-} & = \{x \in V : \bl{e}{x} < 0\} .
\end{align*}
An {\em (open) face} of the arrangement
$\Hf = \{H_{e} : e \in E\}$
is a non-empty intersection of the form
\begin{align*}
\bigcap_{e \in E_{0}} H_{e}
\cap
\bigcap_{e \in E_{+}} H_{e}^{+}
\cap
\bigcap_{e \in E_{-}} H_{e}^{-} ,
\end{align*}
where $E = E_{0} \sqcup E_{+} \sqcup E_{-}$
and $\sqcup$ denotes disjoint union.
For a face $F$ of an arrangement $\Hf$ and a vector~$v$,
let the {\em composition} $F \circ v$ denote the face $G$ of the arrangement $\Hf$
such that $x + \varepsilon \cdot v \in G$
for all $x \in F$ and $\varepsilon > 0$ small enough.
See~\cite[pages~8 and~102]{OM}
where the composition is defined for signed vectors.

\begin{lemma}
Let $U_{1}$ and $U_{2}$ be two chambers
in the restricted arrangement $\Hf_{e}^{\prime\prime}$.
Let $Z_{i}$ be the unique chamber in 
the even restricted arrangement $\Hf_{e}$
containing $U_{i}$.
Then the parities of the two separation sets
$S_{\Hf}(U_{1} \circ e,U_{2} \circ e)$
and
$S_{\Hf_{e}}(Z_{1},Z_{2})$
agree, that is,
\begin{align*}
|S_{\Hf}(U_{1} \circ e, U_{2} \circ e)|
& \equiv
|S_{\Hf_{e}}(Z_{1},Z_{2})|
\bmod 2 .
\end{align*}
\label{lemma_parity_separation}
\end{lemma}
\begin{proof}
We denote by $X$ the set of hyperplanes of $\Hf_{e}$ and
by $S$ the set of vectors $f \in E-\{e\}$ such that
$H_{e}\cap H_{f}$ has even intersection multiplicity. Let
$\iota:S \longrightarrow X$ the map sending $f \in S$ to $H_{e}\cap H_{f}$.
We claim that:
\begin{itemize}
\item[(a)] The set $S_{\Hf}(U_{1} \circ e,U_{2} \circ e) - S$ has even cardinality.
\item[(b)] The image of $S\cap S_{\Hf}(U_{1} \circ e,U_{2} \circ e)$ 
under $\iota$ is $S_{\Hf_{e}}(Z_{1},Z_{2})$.
\item[(c)] The fibers of $\iota$ all have odd cardinality.

\end{itemize}
These three statements immediately imply the result, so it suffices to
prove them.

Consider the equivalence relation $\sim$ on $E-\{e\}$ defined
by $f\sim f'$ if and only if $H_{e}\cap H_{f}=H_{e}\cap H_{f'}$.
Then $S$ is a union of equivalence classes for $\sim$. We claim
that $S_{\Hf}(U_{1} \circ e,U_{2} \circ e)$ is also a union of equivalence classes
for $\sim$. This follows from the fact that $e\not\in S_{\Hf}(U_{1} \circ e,U_{2} \circ e)$
and that if $f\in E-\{e\}$ then $f\in S_{\Hf}(U_{1} \circ e,U_{2} \circ e)$ if and only
if $H_{f}\cap H_{e}$ separates $U_{1}$ and $U_{2}$. In particular, the
set $S_{\Hf}(U_{1} \circ e,U_{2} \circ e)-S$ is also a union of equivalence classes of
$\sim$. But if $f \in E-\{e\}$, then its equivalence class
for $\sim$ is the set of $f'\in E-\{e\}$ such that $H_{f}\cap H_{e}\subset 
H_{f'}$, so it has even cardinality if and only if $f\not\in S$. This
shows that $S_{\Hf}(U_{1} \circ e,U_{2} \circ e)-S$ is a disjoint union of sets of
even cardinality and proves (a). Also, as the fibers of $\iota$
are exactly the equivalence classes of $\sim$ that are contained in
$S$, we obtain (c). We finally prove (b). If $f\in S\cap S_{\Hf}(U_{1} \circ e,
U_{2} \circ e)$ then $U_{1}$ and $U_{2}$ are on opposite sides of
$H_{e}\cap H_{f}$ and $H_{e}\cap H_{f}\in\Hf_{e}$, so $Z_{1}$ and $Z_{2}$ must be on opposite
sides of $H_{e}\cap H_{f}$. Conversely, consider a hyperplane
$H$ of $\Hf_{e}$ such that $Z_{1}$ and $Z_{2}$ are on opposite sides of
$H$. We can write $H = H_{e} \cap H_{f}$ with $f \in S$.
The chambers $U_{1}$ and $U_{2}$ are on opposite sides of the hyperplane $H_{f}$,
so $f \in S_{\Hf}(U_{1} \circ e,U_{2} \circ e)$.
\end{proof}

For $e\in E$, let $Z$ be a chamber of the even restricted arrangement $\Hf_{e}$.
Note that the closure of $Z$ is
a union of closures of
chambers in the restricted arrangement~$\Hf_{e}^{\prime\prime}$,
that is,
$\overline{Z} = \cup_{U \in Q} \overline{U}$
for some subset $Q$ of chambers of~$\Hf_{e}^{\prime\prime}$.
By Lemma~\ref{lemma_parity_separation}
the sign $(-1)^{U \circ e}$ is independent
of $U\in Q$ and hence we define
$(-1)^{Z \circ e} = (-1)^{U \circ e}$ for any $U \in Q$.

\begin{theorem}
Let $\Hf = \{H_{e} : e \in E\}$ be a central hyperplane arrangement
with base chamber~$T_{0}$.
Let $f:V \longrightarrow\C$ be an $L^1$ function.

\begin{itemize}
\item[(i)]
Let $T$ be a chamber of $\Hf$. Then for every $a\in V$, we have
\begin{align}
\int_T f_{a}(x)\:dV-\int_Tf(x)\:dV=
\sum_{U}(-1)^T(-1)^{U\circ e_U}(a,e_U)\int_0^1
\left(\int_U f_{ta}(y)\: dV_{H_{e_U}}\right)\: dt,
\label{equation_reduction_in_integral_form_by_chamber}
\end{align}
where the sum runs over all facets $U$ of $T$ and, for each such
$U$, the vector $e_U$ is the unique element of $E$ such that
$U\subset H_{e_U}$. 

\item[(ii)]
For $e\in E$, let $Z_{0}(e)$ be an arbitrarily chosen 
base chamber of the even restricted arrangement~$\Hf_{e}$.
Let $a$ be a vector in $V$.
Then we have
\begin{align}
P(\Hf, f_{a}) - P(\Hf, f)
& =
2 \cdot
\sum_{e \in E}
(-1)^{Z_{0}(e) \circ e}
\cdot
\bl{a}{e}
\cdot
\int_{0}^{1}
P(\Hf_{e}, \left. f_{ta} \right|_{H_{e}}) dt .
\label{equation_reduction_in_integral_form}
\end{align}

\end{itemize}

\label{theorem_reduction}
\end{theorem}

In particular, if $K$ is a measurable subset of $V$ that has finite
volume, we obtain
\begin{align}
P(\Hf, K + a) - P(\Hf, K)
& =
2 \cdot
\sum_{e \in E}
(-1)^{Z_{0}(e) \circ e}
\cdot
\bl{a}{e}
\cdot
\int_{0}^{1}
P(\Hf_{e}, (K+ta) \cap H_{e})\: dt .
\end{align}

\begin{proof}[Proof of Theorem~\ref{theorem_reduction}.]
For this proof we assume that the vector space $V$ is $\Rrr^{n}$ with
the usual inner product $\bl{\cdot}{\cdot}$.
Hence we write $x = (x_{1}, \ldots, x_{n})$
and $a = (a_{1}, \ldots, a_{n})$.
Suppose first that the function~$f$ is $C^\infty$ with compact support.
Next, let $F:\R^n \longrightarrow\R^n$ be the vector field $F(x) = f(x-ta) \cdot a$.
Note that
$$
\divv(F)
=
\sum_{i=1}^{n} a_{i} \frac{\partial}{\partial x_{i}} f(x-ta)
=
- \frac{\partial}{\partial t} f(x-ta) .
$$
Let $T$ be a chamber of $\Hf$.
The function $t \longmapsto\int_T f_{ta}(x)\:dV$ is also $C^\infty$.
By the Leibniz integral rule and Gauss's divergence theorem
we have
\begin{align}
\label{equation_Gauss}
\frac{d}{dt}\int_T f_{ta}\:dV
& =
\int_T \frac{\partial}{\partial t} f(x-ta) \: dV  \\
& =
\int_T \divv(F) \: dV
\nonumber \\
& =
- 
\int_{\partial T} F \cdot m \: dS ,
\nonumber
\end{align}
where $m$ is the unit normal vector pointing outward from the chamber~$T$.

Let $\Tf^{\prime}$ be the collection of $(n-1)$-dimensional faces
in the arrangement~$\Hf$.
We call the elements in $\Tf^{\prime}$ {\em subchambers}
since they are one dimension
less than that of the chambers in~$\Tf$.
Note that each subchamber $U \in \Tf^{\prime}$ is a subset of exactly
one hyperplane in~$\Hf$.
Let $e_{U}$ denote the unique vector $f$ in $E$ such that $U \subseteq H_{f}$.

The normal vector $m$ at a point in $U \subseteq \partial T$ is $\pm e_{U}$.
Since $m$ points out from $T$ we have
$- (-1)^{T} \bl{a}{m} = (-1)^{U \circ m} \bl{a}{m} = (-1)^{U \circ e_{U}} \bl{a}{e_{U}}$,
where the last equality is true since either both factors change sign or neither does.
We conclude that the last integral
appearing in equation~\eqref{equation_Gauss}
is the sum over all subchambers $U$ of $\Hf$ included in
$\partial T$,
that is,
\begin{align*}
\frac{d}{dt}\int_T f_{ta}\:dV
& =
(-1)^T(-1)^{U\circ e_U} \bl{a}{e_{U}} \int_U f(x-ta) \: dS .
\end{align*}
Also note that the surface differential $dS$ is the natural volume on~$U$.
We deduce equation~\eqref{equation_reduction_in_integral_form_by_chamber} by
integrating both sides, and this proves point~(i) in the case
when $f$ is $C^\infty$ with compact support.

To deduce the general case of (i) from the case of a $C^\infty$
function with compact support, it suffices to prove that
both sides of equation~\eqref{equation_reduction_in_integral_form_by_chamber}
are continuous for the $L^1$ norm, because the space of
$C^\infty$ functions with compact support is dense in the set
of $L^1$ functions for that norm.
This is clear for the left-hand side of
equation~\eqref{equation_reduction_in_integral_form_by_chamber}. 
To show
the continuity of the right-hand side, we need to see that
for every subchamber $U$ of $\Hf$ such that $(a,e_U)\not=0$
the function
$f \longmapsto \int_0^1\int_U f(x-ta) \: dV_{H_{e_U}}(x)dt$ is continuous
for the $L^1$ norm. As $(a,e_U)\not=0$, we know that
$a\not\in H_{e_U}$.
By Fubini's theorem, the double integral
$\int_0^1\int_U f(x-ta) \: dV_{H_{e_U}}(x)dt$ equals, up to a constant,
the integral of $f$ on the set $\{z+ta : z\in U,\ 0\leq t\leq 1\}$
for the measure $V$ on~$\Rrr^{n}$, so the result follows.

We now prove point~(ii). Let $f:V\fl\R$ be an $L^1$ function.
Summing equation~\eqref{equation_reduction_in_integral_form_by_chamber}
over all chambers $T$ of $\Hf$ and noting that each
subchamber $U$ appears in the boundary of exactly two chambers,
we obtain
\begin{align}
P(\Hf,f_a)-P(\Hf,f)
& =
2 \cdot
\sum_{U \in \Tf^{\prime}}
(-1)^{U \circ e_{U}}
\cdot
\bl{a}{e_{U}}
\cdot
\int_0^1\int_U f(x-ta) \: dV_{H_{e_U}}\:dt .
\label{equation_sum_over_subchambers}
\end{align}

To simplify the expression~\eqref{equation_sum_over_subchambers},
consider a subchamber~$U$.
It lies in a unique hyperplane~$H_{e}$,
where $e = e_{U}$.
Furthermore, the subchamber $U$ is contained in a unique chamber~$Z$ of
the even restricted arrangement~$\Hf_{e}$.
Pick a base chamber~$U_{0}(e)$ of the restricted arrangement
$\Hf_{e}^{\prime\prime}$ inside the chamber $Z_{0}(e)$.
We then have by Lemma~\ref{lemma_parity_separation}
\begin{align*}
(-1)^{U \circ e}
=
(-1)^{U_{0}(e) \circ e} \cdot (-1)^{S(U_{0}(e) \circ e, U \circ e)}
=
(-1)^{Z_{0}(e) \circ e} \cdot (-1)^{S(Z_{0}(e), Z)}
=
(-1)^{Z_{0}(e) \circ e} \cdot (-1)^{Z} .
\end{align*}
Hence we can collect terms in the sum in equation~\eqref{equation_sum_over_subchambers}
by first summing over vectors~$e$ in~$E$ and then chambers~$Z$
of the even restricted arrangement~$\Hf_{e}$. This gives that
\begin{align*}
P(\Hf,f_a)-P(\Hf,f)
& =
2 \cdot
\sum_{e \in E}
\sum_{Z \in \Tf(\Hf_{e})}
(-1)^{Z_{0}(e) \circ e} \cdot (-1)^{Z}
\cdot
\bl{a}{e}
\cdot
\int_0^1\int_Z f(x-ta) \: dV_{H_e}\:dt \\
& =
2 \cdot
\sum_{e \in E}
(-1)^{Z_{0}(e) \circ e}
\cdot
\bl{a}{e}
\cdot
\int_0^1 P(\Hf_{e},\left.f_{ta}\right|_{H_{e}})\:dt,
\end{align*}
which is equation~\eqref{equation_reduction_in_integral_form}.
\end{proof}

\begin{table}
$$
\begin{array}{c | l | c | l}
\Hf & \Hf_{e} & \Hf & \Hf_{e} \\ \hline
A_{n} & A_{n-2} &
E_{7} & D_{6} \\
B_{2} & A_{1} &
E_{8} & E_{7} \\
B_{3} & B_{2} \text{ or } A_{1}^2 &
F_{4} & B_{3} \\
B_{n} & B_{n-1} \text{ or } A_{1} \times B_{n-2}, \text{ for } n \geq 4 &
H_{3} & A_{1}^2 \\
D_{4} & A_{1}^{3} &
H_{4} & H_{3} \\
D_{5} & A_{1} \times A_{3} &
I_{2}(2k) & A_{1} \\
D_{n} & A_{1} \times D_{n-2}, \text{ for } n \geq 6 &
I_{2}(2k+1) & \emptyset
\end{array}
$$
\caption{The possible types of even restricted subarrangements for
simple Coxeter arrangements
of dimension $\geq 2$.
For the type $A_{n}$ arrangement
the even restricted subarrangements have type $A_{n-2}$ inside
an $(n-1)$-dimensional space, that is, the even restricted subarrangements are not essential.
Similarly for the odd dihedral arrangement $I_{2}(2k+1)$,
each of the even restricted subarrangements is the empty $1$-dimensional arrangement.}
\label{table_even_subarrangements}
\end{table}

\section{Even restriction sequences}
\label{section_condition_E}

For a sequence $(e_{1}, e_{2}, \ldots , e_{r})$ in the Cartesian power $E^{r}$,
define the subspace
$H_{e_{1},\ldots,e_{r}}$ to be the intersection
$H_{e_{1},\ldots,e_{r}} = \bigcap_{i=1}^{r} H_{e_i}$.
\begin{definition}
An {\em even restriction sequence}
$(e_{1}, e_{2}, \ldots , e_{r})$ and
its associated hyperplane arrangement $\Hf_{e_{1},\ldots,e_{r}}$ on
$H_{e_{1},\ldots,e_{r}}$ is defined recursively by:
\begin{itemize}
\item[(a)]
The empty sequence is an even restriction sequence.
This is the case $r=0$ and we set $\Hf_{\emptyset} = \Hf$.

\item[(b)]
If $r \geq 1$, the sequence 
$(e_{1}, e_{2}, \ldots , e_{r-1})$ is an even restriction sequence
and the subspace $H_{e_{1}, e_{2}, \ldots , e_{r}}$
is a hyperplane in
the arrangement
$\Hf_{e_{1}, e_{2}, \ldots , e_{r-1}}$
then
$(e_{1}, e_{2}, \ldots , e_{r})$ is an even restriction sequence.
Furthermore, let $f\in H_{e_{1}, e_{2}, \ldots , e_{r-1}}$
be a unit normal vector of the hyperplane
$H_{e_{1}, e_{2}, \ldots , e_{r}}$,
that is,
$H_{e_{1}, e_{2}, \ldots , e_{r}}
=
\{x \in H_{e_{1}, e_{2}, \ldots , e_{r-1}} : \bl{f}{x} = 0\}$.
We then set
$\Hf_{e_{1}, e_{2}, \ldots , e_{r}}$ to be
the even restricted arrangement
$\left(\Hf_{e_{1}, e_{2}, \ldots , e_{r-1}}\right)_{f}$
in the subspace
$H_{e_{1}, e_{2}, \ldots , e_{r}}$.
\end{itemize}
Finally, let $P_{r} \subseteq E^{r}$ denote the set of all
even restriction sequences of length $r$.
\end{definition}

We introduce the following definitions.

\begin{definition}
Let $\Hf = \{H_{e} : e \in E\}$ be a hyperplane arrangement.
\begin{itemize}
\item[(a)] We say that an even restriction sequence
$(e_{1},\ldots,e_{r})\in P_r$ is \emph{maximal} if
it cannot be extended to a longer even restriction sequence.
More formally, 
there is no $e_{r+1}\in E$ such that $(e_{1},\ldots,e_{r},e_{r+1})$ is an
even restriction sequence. We denote by $P$ the set of
maximal even restriction sequences
of the arrangement $\Hf$.

\item[(b)]
We say that the arrangement $\Hf$ is \emph{even} if
for every $r\geq 0$ and every even restriction sequence
$(e_{1},\ldots,e_{r})\in P_{r}$
the arrangement $\Hf_{e_{1},\ldots,e_{r}}$ is essential.

\end{itemize}
\label{def_even}
\end{definition}

\begin{lemma}
Let $\Hf$ be an $n$-dimensional hyperplane arrangement.
\begin{itemize}
\item[(i)]  If $(e_{1},\ldots,e_{r})\in P_{r}$ then $\dim(H_{e_{1},\ldots,e_{r}})=
n-r$ (in particular, the vectors $e_{1},\ldots,e_{r}$ are linearly independent)
and every element of $\Hf_{e_{1},\ldots,e_{r}}$ is an intersection of
hyperplanes of $\Hf$.

\item[(ii)] We have $P_{r}=\varnothing$ for $r>n$.

\item[(iii)] The following statements are equivalent:
\begin{itemize}
\item[(a)]
The arrangement $\Hf$ is even.
\item[(b)]
Every even restriction sequence
can be extended to an even restriction sequence
of length~$n$, that is,
for every $0\leq r\leq n$ and every $(e_{1},\ldots,e_{r})\in P_{r}$,
there exist $e_{r+1},\ldots,e_{n}\in E$ such
that $(e_{1},e_{2},\ldots,e_{n})\in P_{n}$.
\item[(c)] We have $P=P_n$.
\end{itemize}
\end{itemize}
\label{lemma_nested_even_arrangements}
\end{lemma}
\begin{proof}
If $1\leq i \leq r$ then by definition of $H_{e_{1},\ldots,e_i}$
we have $\dim(H_{e_{1},\ldots,e_i})=\dim(H_{e_{1},\ldots,e_{i-1}})-1$. This
implies the first statement of~(i). We prove the second statement of~(i)
by induction on $r$. It is clear if $r=0$, so assume that $r\geq 1$ and
that the conclusion holds for $\Hf_{e_{1},\ldots,e_{r-1}}$. As
$\Hf_{e_{1},\ldots,e_{r}}$ is a subarrangement of the restriction of
$\Hf_{e_{1},\ldots,e_{r-1}}$ to the hyperplane $H_{e_{1},\ldots,e_{r}}=H_{e_{1},\ldots,
e_{r-1}}\cap H_{e_{r}}$ of $H_{e_{1},\ldots,e_{r-1}}$, each element of
$\Hf_{e_{1},\ldots,e_{r}}$ is the intersection of $H_{e_{r}}$ and an element
of $\Hf_{e_{1},\ldots,e_{r-1}}$. This implies~(i).

Statement~(ii) follows immediately from (i).

We now prove statement~(iii).
Suppose that (a) holds. We prove
(b) by descending induction on $r$. If $r=n$, there is nothing to
prove. Suppose that $r<n$ and that we know the statement for~$r+1$,
and let $(e_{1},\ldots,e_{r})\in P_{r}$. As the arrangement
$\Hf_{e_{1},\ldots,e_{r}}$ is an essential arrangement in the
$(n-r)$-dimensional space $H_{e_{1},\ldots,e_{r}}$, and as $n-r>0$, this
arrangement has at least one hyperplane.
By the second statement of (i), there exists
$e_{r+1}\in E$ such that $(e_{1},\ldots,e_{r+1}) \in P_{r+1}$. We finish the proof
by applying the induction hypothesis to $(e_{1},\ldots,e_{r+1})$.

Conversely, suppose that (b) holds, and let $(e_{1},\ldots,e_{r})\in P_{r}$,
with $0\leq r\leq n$. By~(b), there exist vectors $e_{r+1},\ldots,e_{n} \in E$
such that $(e_{1},\ldots,e_{n})\in P_{n}$. 
By~(i), we know that $H_{e_{1},\ldots,e_{n}}=\{0\}$ and that
$H_{e_{1},\ldots,e_{n}}$ is an intersection
of hyperplanes of $\Hf_{e_{1},\ldots,e_{r}}$.
This implies that $\Hf_{e_{1},\ldots,e_{r}}$ is essential, and hence~(a) holds.

Finally, in light of statement (ii), it is clear that (b) and (c)
are equivalent.
\end{proof}

\begin{definition}
If $(e_{1},\ldots,e_{r})\in P$ then the arrangement $\Hf_{e_{1},\ldots,e_{r}}$
is empty, so it only has one chamber $Z=H_{e_{1},\ldots,e_{r}}$.
By the discussion before Theorem~\ref{theorem_reduction}, we have
a well-defined sign $(-1)^{Z\circ e_{r} \circdots e_{1}}$,
and we denote this sign by $(-1)^{e_{r} \circdots e_{1}}$.
\label{sign_maxmimal_even_restriction_sequence}
\end{definition}

\begin{remark}
{\rm
\begin{itemize}
\item[(1)] Let $(e_{1},\ldots,e_{r})\in P$. Then $(-1)^{e_{r} \circdots e_{1}}$
is the sign of any chamber of $\Hf$ that contains the vector
$v+\epsilon e_{r}+\epsilon^2 e_{r-1}+\cdots+\epsilon^r e_{1}$, for
$v\in H_{e_{1},\ldots,e_{r}}$ nonzero and $\epsilon>0$ small enough.

\item[(2)] Suppose that $\Hf$ is an even arrangement. Then $P=P_n$ by
Lemma~\ref{lemma_nested_even_arrangements}. So for any
$(e_{1},\ldots,e_{n})\in P_n$, we have a sign $(-1)^{e_{n} \circdots e_{1}}$,
which is the sign of the chamber of $\Hf$ containing the vector
$\epsilon e_{n} + \epsilon^2 e_{n-1} + \cdots + \epsilon^n e_{1}$ for
$\epsilon>0$ small enough.
\end{itemize}
}
\end{remark}

For $r$ a nonnegative integer, we introduce the following equivalence
relation on the Cartesian power $E^{r}$: $(e_{1},\ldots,e_{r})\sim(f_{1},\ldots,f_{r})$ if and
only if $H_{e_{1},\ldots,e_i}=H_{f_{1},\ldots,f_i}$ for every $0 \leq i \leq r$.
The following facts are straightforward consequences of the definition of this equivalence relation.
\begin{itemize}
\item[(1)]
If $0\leq i\leq r$,  $(e_{1},\ldots,e_{r}) \in E^{r}$
and $(f_{1},\ldots,f_i) \in E^{i}$ then $(e_{1},\ldots,e_i)\sim
(f_{1},\ldots,f_i)$ if and only if $(e_{1},\ldots,e_{r})\sim(f_{1},\ldots,f_i,e_{i+1},
\ldots,e_{r})$.

\item[(2)]
If $(e_{1},\ldots,e_{r}) \sim (f_{1},\ldots,f_{r})$
then $(e_{1},\ldots,e_{r}) \in P_{r}$ if and only if $(f_{1},\ldots,f_{r})\in P_{r}$.
If this condition holds then $\Hf_{e_{1},\ldots,e_{r}}=\Hf_{f_{1},\ldots,f_{r}}$.

\item[(3)] 
If $(e_{1},\ldots,e_{r}) \sim (f_{1},\ldots,f_{r})$
then $(e_{1},\ldots,e_{r}) \in P$ if and only if $(f_{1},\ldots,f_{r})\in P$.
If this condition holds then $(-1)^{e_{r} \circdots e_{1}}=
(-1)^{f_{r} \circdots f_{1}}$.
\end{itemize}

We now specialize these notions to the case of Coxeter arrangements.
Suppose that $\Hf$ is a Coxeter arrangement and let $r$ be a
nonnegative integer.
Then we denote by $E_{r}^{(0)} \subseteq E^{r}$ 
the set of sequences $(e_{1},\ldots,e_{r})$
of pairwise orthogonal elements of $E$.

\begin{proposition}
Suppose that $\Hf = (H_{e})_{e \in E}$ is a Coxeter arrangement
and let $r$ be a nonnegative integer.
\begin{itemize}
\item[(i)] Let $(e_{1},\ldots,e_{r})\in P_{r}$.
Then $\Hf_{e_{1},\ldots,e_{r}}$ is a Coxeter arrangement
on $H_{e_{1},\ldots,e_{r}}$ given by the 
finite set of vectors $E \cap H_{e_{1},\ldots,e_{r}}$.

\item[(ii)] If $(e_{1},\ldots,e_{r}),(f_{1},\ldots,f_{r})\in E_{r}^{(0)}$
and $(e_{1},\ldots,e_{r})\sim(f_{1},\ldots,f_{r})$ then the
equality $(e_{1},\ldots,e_{r})$ $=$ $(f_{1},\ldots,f_{r})$ holds.

\item[(iii)] Let $(e_{1},\ldots,e_{r})\in E^{r}$. Then $(e_{1},\ldots,e_{r})$ is
in $P_{r}$ if and only if there exists $(f_{1},\ldots,f_{r})\in E_{r}^{(0)}$
such that $(e_{1},\ldots,e_{r})\sim(f_{1},\ldots,f_{r})$.
Furthermore, this sequence $(f_{1},\ldots,f_{r})$ is necessarily unique by~(ii).

In particular, we have $E_{r}^{(0)} \subseteq P_{r}$, and this
inclusion induces a bijection $E_{r}^{(0)}\simeq P_r/\sim$.

\item[(iv)] Let $(e_{1},\ldots,e_{r})\in E_{r}^{(0)}$. Then
$(e_{1},\ldots,e_{r})\in P$ if and only if $E\cap\{e_{1},\ldots,e_{r}\}^\perp=
\{0\}$.

\item[(v)] There exists an integer $r$ with $0 \leq r \leq n$ such that
$P=P_r$. 

\end{itemize}
\label{proposition_nested_even_arrangements_Coxeter}
\end{proposition}
\begin{proof}
We prove (i) by induction on $r$. The result is clear if $r=0$.
So assume that $r\geq 1$ and that we know the result for
$r-1$, and let $(e_{1},\ldots,e_{r}) \in P_{r}$. By the induction hypothesis,
the arrangement $\Hf_{e_{1},\ldots,e_{r-1}}$ is a Coxeter arrangement, and
its hyperplanes are the codimension one subspaces
orthogonal to the vectors in $E \cap H_{e_{1},\ldots,e_{r-1}}$. 
By Proposition~\ref{prop_even_arrangement_Coxeter},
the arrangement $\Hf_{e_{1},\ldots,e_{r}}$ is
also Coxeter and
the elements of $\Hf_{e_{1},\ldots,e_{r}}$ are exactly the intersections
$H_{e_{1},\ldots,e_{r}}\cap H_{f}$, where $f\in E\cap H_{e_{1},\ldots,e_{r-1}}$ is
such that $\bl{f}{e_{r}} = 0$. In other words, they are the intersections
$H_{e_{1},\ldots,e_{r}}\cap H_{f}$, with $f\in H_{e_{1},\ldots,e_{r}}$. This finishes
the proof.

We prove (ii) by induction on $r$. If $r \leq 1$ the
result is clear, so assume that $r\geq 2$ and that we know
the result for $r-1$. Let $(e_{1},\ldots,e_{r}),(f_{1},\ldots,f_{r})\in
E_{r}^{(0)}$ such that $(e_{1},\ldots,e_{r})\sim(f_{1},\ldots,f_{r})$. Then
$(e_{1},\ldots,e_{r-1}),(f_{1},\ldots,f_{r-1})\in E_{r-1}^{(0)}$ and
$(e_{1},\ldots,e_{r-1})\sim(f_{1},\ldots,f_{r-1})$, so $e_i=f_i$ for
$1\leq i\leq r-1$ by the induction hypothesis. The assumption that
$H_{e_{1},\ldots,e_{r}}=H_{e_{1},\ldots,e_{r-1},f_{r}}$ shows that
the families $(e_{1},\ldots,e_{r})$ and $(e_{1},\ldots,e_{r-1},f_{r})$ span
the same subspace of $V$ and, as $e_{r},f_{r}\in\Span(e_{1},\ldots,e_{r-1})^{\perp}$,
we conclude that $e_{r}$ and $f_{r}$ are colinear, hence equal. 

We prove (iii), again by induction on $r$. If $r\leq 1$ then
$E^{r}=P_{r}=E_{r}^{(0)}$ and the result follows. Suppose that $r\geq 2$
and that we know the result for $r-1$. Let $(e_{1},\ldots,e_{r})\in E^{r}$.
Suppose that $(e_{1},\ldots,e_{r})\in P_{r}$.
Then $(e_{1},\ldots,e_{r-1}) \in P_{r-1}$, so by the induction hypothesis and observations~(1) and~(2) above
we may assume that
$(e_{1},\ldots,e_{r-1})\in E_{r-1}^{(0)}$. As $H_{e_{1},\ldots,e_{r}}$ is an element
of $\Hf_{e_{1},\ldots,e_{r-1}}$, there exists by (i) an element
$f$ of $E\cap H_{e_{1},\ldots,e_{r-1}}$ such that $H_{e_{1},\ldots,e_{r}}=H_{f}\cap
H_{e_{1},\ldots,e_{r-1}}$. Then $(e_{1},\ldots,e_{r-1},f)\in E_{r}^{(0)}$ and
$(e_{1},\ldots,e_{r-1},f)\sim(e_{1},\ldots,e_{r})$, so we are done.
Now suppose that there exists $(f_{1},\ldots,f_{r})\in E_{r}^{(0)}$
such that $(e_{1},\ldots,e_{r})\sim(f_{1},\ldots,f_{r})$. By observation
(2) above, it suffices to show that $(f_{1},\ldots,f_{r})\in P_{r}$, so we may
assume that $(e_{1},\ldots,e_{r})\in E_{r}^{(0)}$. By the induction hypothesis,
we have $(e_{1},\ldots,e_{r-1})\in P_{r-1}$ and by~(i) the hyperplane
$H_{e_{1},\ldots,e_{r}}$ is an element of $\Hf_{e_{1},\ldots,e_{r-1}}$, which shows
that $(e_{1},\ldots,e_{r})\in P_{r}$.

Point (iv) immediately follows from point (iii) and from the
definition of $P$.

To prove (v), we need to see that all maximal subsets of pairwise
orthogonal elements of $E$ have the same cardinality, which is
a well-known fact. For crystallographic root systems
this follows, for example, from Lemma~5.7 
of~\cite{Herb-DSC} and from the fact that all $2$-structures are
conjugated under the Weyl group, as asserted before Theorem~5.5
of~\cite{Herb-DSC}; for the general case, see Appendix~B
of~\cite{Ehrenborg_Morel_Readdy}.
\end{proof}

\begin{corollary}
Suppose $\Hf$ is a Coxeter arrangement with Coxeter group~$W$
in an $n$-dimensional vector space~$V$. Then the following four statements
are equivalent:
\begin{itemize}
\item[(a)]
The arrangement $\Hf$ is even.
\item[(b)]
Every sequence of pairwise orthogonal elements of $E$
can be extended to a sequence of pairwise orthogonal elements of length $n$,
that is,
for $0\leq r\leq n$ and $(e_{1},\ldots,e_{r}) \in E_{r}^{(0)}$, there exist
vectors $e_{r+1},\ldots,e_{n}\in E$ such that $(e_{1},\ldots, e_{n}) \in E_{n}^{(0)}$.
\item[(c)]
The map $-\id_V$ belongs to the Coxeter group~$W$.
\item[(d)]
The arrangement $\Hf$ is a Cartesian product of arrangements of type
$A_{1}$, $B_{n}$ for $n\geq 2$, $D_{n}$ for $n\geq 4$ even,
$E_{7}$, $E_{8}$,
$F_{4}$,
$H_{3}$, $H_{4}$ and 
$I_{2}(2k)$ for $k\geq 2$.
\end{itemize}
\label{corollary_nested_even_arrangements_Coxeter}
\end{corollary}

The equivalence of (b), (c) and (d) is well-known.
For completeness we include a proof.

\begin{proof}[Proof of Corollary~\ref{corollary_nested_even_arrangements_Coxeter}.]
The equivalence of (a) and (b) follows from
Lemma~\ref{lemma_nested_even_arrangements} (iii)
by using statements~(ii) and~(iii) of 
Proposition~\ref{proposition_nested_even_arrangements_Coxeter}.

We prove that (b) and (d) are equivalent. As both conditions hold
for a Cartesian product of two arrangements if and only if they hold
for each arrangement, we may assume that $\Hf$ is a simple
Coxeter arrangement. 
If $\Hf$ is of type~$A_{n}$ for $n\geq 2$, $D_{n}$ for $n\geq 5$ odd, $I_{2}(m)$
for $m\geq 3$ odd or $E_{6}$, then $\Phi$ does not contain a family of
$n$ pairwise orthogonal pseudo-roots, so a fortiori condition (b) does
not hold. If $\Hf$ is of type~$A_{1}$, $B_{2}$ or $I_{2}(2k)$ with $k\geq 2$, then
condition (b) is clear.
Suppose that $\Hf$ is of type~$B_{n}$ with
$n\geq 3$, $D_{n}$ with $n\geq 4$ even, $E_{7}$, $E_{8}$, $F_{4}$, $H_{3}$ or $H_{4}$.
It suffices to show that if $e \in E$ then the arrangement on~$H_{e}$
given by the vectors of $E\cap H_{e}$, which is $\Hf_{e}$ by
statement~(i) of 
Proposition~\ref{proposition_nested_even_arrangements_Coxeter},
satisfies condition (b). This follows
from a straightforward induction
using Table~\ref{table_even_subarrangements}.

Suppose that $\Hf$ satisfies (b).
Then there exists
a family of pairwise orthogonal pseudo-roots
$(\alpha_{1},\ldots,\alpha_{n})\in\Phi$;
in particular, this family generates $V$.
Denote by $s_\alpha$ the reflection corresponding to a root $\alpha$.
We obtain that $-\id_V=s_{\alpha_{1}} \cdots s_{\alpha_{n}}\in W$, which is (c).

Finally, we show by induction on $n$ that (c) implies (a). 
If $n=1$, then (c) implies that $\Hf$ is not empty, hence even.
Suppose that $n\geq 2$ and that we know the result for Coxeter
arrangements on vector spaces of dimension $n-1$. 
Let $e\in E$ and set $W_e=\{w\in W : w(e)=e\}$.
By a theorem of Steinberg~\cite[Theorem~1.5]{Steinberg},
the group $W_{e}$ is a reflection
subgroup of $W$, hence it is generated by reflections $s_f$,
for $f\in E$, where $s_f$ is as before the orthogonal reflection in
the hyperplane~$H_f$.
As $s_f(e)=e$ if and only if $(e,f)=0$, the group
$W_e$ is exactly the subgroup of $W$ generated by the reflections
$s_f$ for $(e,f)=0$, in other words, it is the Coxeter group of
the even restricted arrangement $\Hf_e$. As $-\id_V\in W$, we have
$-s_e\in W$ and, as $-s_e(e)=e$, we conclude that $-s_e\in W_e$.
But $-s_e$ acts by $-\id_{H_e}$ on the hyperplane $H_e$, so we conclude
that the Coxeter arrangement $\Hf_e$ satisfies condition (c), hence
that it is even by the induction hypothesis. So we have shown
that $\Hf_e$ is even for every $e\in E$, which implies that $\Hf$ is
even.
\end{proof}

\section{Sufficiently symmetric sets}
\label{section_sufficiently_symmetric}

We now turn our attention to a collection of measurable sets
that we call sufficiently symmetric 
for which the evaluation of the pizza quantity is easier.
We conclude the section by showing
that for Coxeter arrangements a measurable set
of finite volume that is stable with respect to the
reflections in the arrangement is sufficiently symmetric.

\begin{definition}
For a nonempty $n$-dimensional arrangement $\Hf$,
we say that a measurable subset of $V$ with finite volume
$K$ is \emph{sufficiently symmetric} with respect to $\Hf$ if,
for every $0 \leq r \leq n-1$, for every even restriction sequence
$(e_{1},\ldots,e_{r}) \in P_{r}$ that is not maximal,
and for every $a \in H_{e_{1},\ldots,e_{r}}^{\perp}$, 
the pizza quantity
$P(\Hf_{e_{1},\ldots,e_{r}},(K+a)\cap H_{e_{1},\ldots,e_{r}})$
vanishes.
\label{definition_sufficiently_symmetric}
\end{definition}

We can check this condition inductively.

\begin{lemma}
Let $\Hf = (H_{e})_{e \in E}$ be a nonempty hyperplane arrangement
and let $K$ be a measurable subset of $V$ with finite volume.
Then the following conditions are equivalent:
\begin{itemize}
\item[(a)]
The set $K$ is sufficiently symmetric with respect to $\Hf$.
\item[(b)]
The pizza quantity $P(\Hf,K)$ is equal to zero and
for every $e \in E$ and for every $b \in H_{e}^{\perp}$,
the set $(K+b) \cap H_{e}$ is sufficiently
symmetric with respect to~$\Hf_{e}$.
\end{itemize}
\label{lemma_sufficiently_symmetric}
\end{lemma}
\begin{proof}
Suppose that (a) holds. As $\Hf$ is nonempty, the empty even restriction
sequence is not maximal, and
we conclude that $P(\Hf,K)=0$ by taking $r=0$ in
Definition~\ref{definition_sufficiently_symmetric}.
Now let $e\in E$ and $b\in H_{e}^{\perp}$.
Let $1 \leq r \leq n-1$ and $e_{2},\ldots,e_{r}\in E$ such that
$(e,e_{2},\ldots,e_{r})\in P_{r} - P$, and let $a\in H_{e}\cap
H_{e,e_{2},\ldots,e_{r}}^{\perp}$.
We wish to show that $P(\Hf_{e,e_{2},\ldots,e_{r}},L)=0$ where
$L=(((K+b)\cap H_{e})+a)\cap
H_{e,e_{2},\ldots,e_{r}}=(K+a+b)\cap H_{e,e_{2},\ldots,e_{r}}$.
This follows from the hypothesis on $K$ because
$b\in H_{e}^{\perp}\subset H_{e,e_{2},\ldots,e_{r}}^{\perp}$.

Suppose that (b) holds. Let $0 \leq r \leq n-1$, let $(e_{1},\ldots,e_{r}) \in P_{r} - P$
and let $a \in H_{e_{1},\ldots,e_{r}}^{\perp}$.
We wish to show that
$P(\Hf_{e_{1},\ldots,e_{r}},(K+a)\cap H_{e_{1},\ldots,e_{r}})=0$.
If $r=0$  then $\Hf_{e_{1},\ldots,e_{r}}=\Hf$,
$H_{e_{1},\ldots,e_{r}}=V$ and $a \in V^\perp=\{0\}$, 
so the desired
result follows from the hypothesis. Suppose that $r \geq 1$. 
We write $a = \lambda e_{1}+b$, with
$\lambda\in\R$ and $b\in e_{1}^{\perp}\cap H_{e_{1},\ldots,e_{r}}^{\perp}=
H_{e_{1}}\cap H_{e_{1},\ldots,e_{r}}^{\perp}$. Let
$L=(K+\lambda e_{1})\cap H_{e_{1}}$.
By condition~(b) the set $L$ is sufficiently
symmetric with respect to~$\Hf_{e_{1}}$, so
$P(\Hf_{e_{1},\ldots,e_{r}},(L+b)\cap H_{e_{1},\ldots,e_{r}})=0$.
As
$(K+a)\cap H_{e_{1},\ldots,e_{r}} = (L+b)\cap H_{e_{1},\ldots,e_{r}}$,
we are done.
\end{proof}

For Coxeter arrangements there is a more natural condition on
measurable sets, namely, that of being stable by the action of the Coxeter
group. We verify this also behaves well in the even restricted
arrangement setting.

\begin{proposition}
Let $\Hf$ be a Coxeter arrangement
in the vector space~$V$
with Coxeter group~$W$.
Let $L$ be a measurable subset of $V$
stable under the action of~$W$.
Let $H_{e}$ be a hyperplane in~$\Hf$
and let $b$ be a scalar multiple of the vector $e$.
Then the measurable subset $(L + b) \cap H_{e}$ in the space~$H_{e}$ is
stable by the action of the Coxeter group
of the even restricted arrangement~$\Hf_{e}$.
\label{proposition_missing}
\end{proposition}

Remember that the arrangement $\Hf_{e}$ is also a Coxeter arrangement
by Proposition~\ref{prop_even_arrangement_Coxeter}.

\begin{proof}[Proof of Proposition~\ref{proposition_missing}.]
It is suffices to show that
$(L + b) \cap H_{e}$ is stable under
any reflection in $V'$, where $V'$ is
a hyperplane in~$\Hf_{e}$,
that is,
$V'$ is a codimension $2$ subspace of $\Hf$
with even intersection multiplicity
and is contained in~$H_{e}$.
Fix such a subspace $V'$.
By Proposition~\ref{prop_even_arrangement_Coxeter}
there exists a vector $f \in E$ such that $\bl{e}{f} = 0$ and $V'=H_{e}\cap H_{f}$.
Since the vectors $e$ and $f$ are orthogonal,
the vector~$b$ lies in the hyperplane~$H_{f}$.
Since $L$ is stable under the reflection in~$H_{f}$,
so is the translate $L + b$.
Note that reflecting
$(L + b) \cap H_{e}$ in~$V'$
(this takes place in~$H_{e}$)
is equivalent to reflecting
$(L + b) \cap H_{e}$ in~$H_{f}$.
Hence 
$(L + b) \cap H_{e}$
is stable under the action of
the Coxeter group of~$\Hf_{e}$.
\end{proof}

\begin{corollary}
Suppose that $\Hf$ is a nonempty Coxeter arrangement and
that $K$ is a measurable subset of $V$ of finite volume
that is stable under the action of the elements of the
Coxeter group $W$ of $\Hf$. Then $K$ is sufficiently symmetric with
respect to the arrangement~$\Hf$. 
\label{corollary_sufficiently_symmetric_Coxeter}
\end{corollary}
\begin{proof}
We prove the result by induction on $n = \dim V$. 
If $n=1$ then we have to show that
$P(\Hf,K)=0$. This follows from the hypothesis and
from Corollary~\ref{corollary_Coxeter_on_a_hyperplane}.
This corollary applies because we assume
that the arrangement $\Hf$ is not empty. Suppose that
$n\geq 2$ and that we know the result for $n-1$. 
If $e\in E$ and $b\in H_{e}^{\perp}$ then the set $(K+b)\cap H_{e}$ is stable
by the Coxeter group of the arrangement $\Hf_{e}$ by
Proposition~\ref{proposition_missing}, hence sufficiently symmetric
with respect to $\Hf_{e}$ by the induction hypothesis.
The result now follows from Lemma~\ref{lemma_sufficiently_symmetric}.
\end{proof}

\section{Evaluating the pizza quantity}
\label{section_pizza_is_polynomial}

We now evaluate the pizza quantity for sufficiently symmetric
measurable sets, and deduce that it is zero for even
Coxeter arrangements and sufficiently symmetric measurable sets of finite volume.

\begin{theorem}
Suppose that $K$ is a measurable subset of $V$ with finite volume,
and that it is sufficiently symmetric with respect to $\Hf$. Then
for every $a\in V$ we have
\begin{align*}
P(\Hf,K+a)
= &
\sum_{(e_{1},\ldots,e_{r})\in P/\sim}2^r (-1)^{e_{r} \circdots e_{1}}
(a,e_{1})(\pi_{e_{1}}(a),e_{2}) \cdots
(\pi_{e_{1},\ldots,e_{r-1}}(a),e_{r})\\
&
\hspace*{10 mm} \cdot
\int_{0}^{1}\int_{0}^{t_1} \cdots \int_{0}^{t_{r-1}}\Vol_{H_{e_{1},\ldots,e_{r}}}
(K_{e_{1},\ldots,e_{r}}(t_1,\ldots,t_r))
dt_{r} \cdots dt_{2} dt_{1} ,
\end{align*}
where, for every $(e_{1},\ldots,e_{r})\in E^r$, we denote by
$\pi_{e_{1},\ldots,e_{r}}$ the orthogonal projection on $H_{e_{1},\ldots,e_{r}}$ and
by $K_{e_{1},\ldots,e_{r}}(t_1,\ldots,t_r)$ the measurable subset
\[H_{e_{1},\ldots,e_{r}}\cap(K+
t_1(a-\pi_{e_{1}}(a))+t_2(\pi_{e_{1}}(a)-\pi_{e_{1},e_{2}}(a))
+ \cdots +
t_r(\pi_{e_{1},\ldots,e_{r-1}}(a)-\pi_{e_{1},\ldots,e_{r}}(a)))\]
of $H_{e_{1},\ldots,e_{r}}$.
\label{theorem_calculation_pizza_quantity}
\end{theorem}
\begin{proof}
Iterate Theorem~\ref{theorem_reduction}
in the following form:
$P(\Hf, K + a)
=
P(\Hf, K)
+
\cdots$.
At each step use the sufficiently symmetric condition to 
note that the $P(\Hf, K)$ terms vanish.
\end{proof}

\begin{remark}
{\rm 
We could use other projections instead of the orthogonal projections
$\pi_{e_{1},\ldots,e_{r}}$. For Coxeter arrangements, it is natural to use
orthogonal projections, but for other arrangements a different choice
might be more appropriate. We keep the discussion to
the orthogonal projections because there is
no canonical general choice and 
the statements are more straightforward.
}
\end{remark}

\begin{theorem}
Suppose that $\Hf$ is an even hyperplane arrangement, and define
the homogenous degree $n$ polynomial function $f_\Hf:V \longrightarrow \R$ by
\[f_\Hf(a)=
\frac{2^n}{n!}
\sum_{(e_{1},\ldots,e_{n})\in P_n/\sim}(-1)^{e_{r} \circdots e_{1}}
(a,e_{1}) (\pi_{e_{1}}(a),e_{2})
\cdots
(\pi_{e_{1},\ldots,e_{n-1}}(a),e_{n}).\]
Then for any measurable subset $K$ of $V$ of finite
volume that is sufficiently symmetric with respect to $\Hf$, and for
every $a\in V$ such that
\[-(t_1(a-\pi_{e_{1}}(a))+t_2(\pi_{e_{1}}(a)-\pi_{e_{1},e_{2}}(a))
+ \cdots +
t_r(\pi_{e_{1},\ldots,e_{r-1}}(a)-\pi_{e_{1},\ldots,e_{r}}(a)))\in K\]
for $0\leq t_r\leq t_{r-1}\leq\cdots\leq t_1\leq 1$,
we have
\[P(\Hf,K+a)=f_\Hf(a).\]
\label{theorem_pizza_is_polynomial}
\end{theorem}

\begin{remark}
{\rm
The conditions of Theorem~\ref{theorem_pizza_is_polynomial}
on $K$ and $a$ hold in the following cases:
\begin{itemize}
\item[(1)] The set $K$ is convex
(of finite volume), sufficiently symmetric with respect to $\Hf$
and $0\in K+a$.
\item[(2)] The arrangement $\Hf$ is Coxeter with Coxeter group
$W$, and the set $K$ is stable by $W$ and contains the
convex hull of the finite set $\{w(-a) : w\in W\}$
(as $W$ contains $-\id_V$ because $\Hf$ is even, it is equivalent to
say that $K$ contains the convex hull of the set $\{w(a) : w\in W\}$).
\item[(3)] The arrangement $\Hf$ is Coxeter, the set $K$ is convex
(of finite volume) and stable by its Coxeter group, and
$0\in K+a$.

\end{itemize}
}
\end{remark}

\begin{proof}[Proof of Theorem~\ref{theorem_pizza_is_polynomial}.]
As $\Hf$ is even, we have $P/\sim=P_n/\sim$. Moreover, by the conditions
on $K$ and $a$,
for every $(e_{1},\ldots,e_{r})\in P_r$ and every
$(t_1,\ldots,t_r)\in\Rrr^r$ such that $0\leq t_r\leq t_{r-1}\leq\cdots\leq
t_1\leq 1$, the subset $K_{e_{1},\ldots,e_{r}}(t_1,\ldots,t_r)$ of $H_{e_{1},\ldots,e_{n}}$ contains $0$. Thus, if $r=n$, we obtain that
$K_{e_{1},\ldots,e_{n}}(t_1,\ldots,t_n)=H_{e_{1},\ldots,e_{n}}=\{0\}$, and hence
\[\Vol_{H_{e_{1},\ldots,e_{n}}}(K_{e_{1},\ldots,e_{n}}(t_1,\ldots,t_n))=1.\]
The corollary then follows from Theorem~\ref{theorem_calculation_pizza_quantity}
and from the fact that
\begin{align*}
\int_{0}^{1}\int_{0}^{t_1} \cdots \int_{0}^{t_{n-1}}
dt_{n} \cdots dt_{2} dt_{1}
& =
\Vol(\{(t_{1},t_{2}, \ldots, t_{n}) : 0 \leq t_{n} \leq \cdots \leq t_{2} \leq t_{1} \leq 1\})
=
\frac{1}{n!}.
\qedhere
\end{align*}
\end{proof}

\begin{corollary}
Suppose that $\Hf$ is a Coxeter arrangement, and let $0 \leq r \leq n$
be the integer such that $P=P_r$. Let $K$ be a measurable subset
of $V$ of finite volume that is sufficiently symmetric with respect
to $\Hf$ (for example, $K$ could be stable by the Coxeter group of $\Hf$).
Then for every $a\in V$ we have
\begin{align*}
P&(\Hf,K+a)=2^r
\sum_{(e_{1},\ldots,e_{r})\in E_{r}^{(0)}}(-1)^{e_{r} \circdots e_{1}}
(a,e_{1})(a,e_{2}) \cdots (a,e_{r})\\
&\int_{0}^{1} \int_{0}^{t_1} \cdots \int_{0}^{t_{r-1}}\Vol_{H_{e_{1},\ldots,e_{r}}}
(H_{e_{1},\ldots,e_{r}}\cap
(K\cap t_1(a,e_{1})+t_2(a,e_{2})+\cdots+t_r(a,e_{r})))
dt_{r} \cdots dt_{2} dt_{1}.
\end{align*}
\end{corollary}

\begin{remark}
{\rm
If $K$ is the ball
$\Ball(0,R)$ and $\|a\|\leq R$ then for every $(e_{1},\ldots,e_{r})\in P$ and
$(t_1,\ldots,t_r)\in\Rrr^r$ such that
$0 \leq t_r \leq \cdots \leq t_1 \leq 1$,
the set $K_{e_{1},\ldots,e_{r}}(t_1,\ldots,t_r)$ of
Theorem~\ref{theorem_calculation_pizza_quantity} is a ball of
radius $\sqrt{R(t_1,\ldots,t_r)}$, where
$$
R(t_1,\ldots,t_r)
=
R^{2} - t_1^2\|a-\pi_{e_{1}}(a)\|^2 - t_2^2\|\pi_{e_{1}}(a)-\pi_{e_{1},e_{2}}(a)\|^2
- \cdots
- t_r^2\|\pi_{e_{1},\ldots,e_{r-1}}(a)-\pi_{e_{1},\ldots,e_{r}}(a)\|^2.
$$
So we can conclude the following corollary.
}
\label{remark_r_R_R}
\end{remark}
\begin{corollary}
Suppose that the ball $\Ball(0,R)$ is sufficiently symmetric with respect
to $\Hf$. Then for every $R \geq 0$ and
every $a\in V$ such that $\|a\|\leq R$, we have
\begin{align*}
P(\Hf,\Ball(a,R)) 
= & \:
\frac{2^r\pi^{{r}/{2}}}{\Gamma({r}/{2}+1)} \\
&
\cdot
\sum_{(e_{1},\ldots,e_{r})\in P/\sim}
(-1)^{e_{r} \circdots e_{1}}
(a,e_{1}) (\pi_{e_{1}}(a),e_{2}) \cdots
(\pi_{e_{1},\ldots,e_{r-1}}(a),e_{r}) \\
&
\cdot
\int_{0}^{1}\int_{0}^{t_1} \cdots \int_{0}^{t_{r-1}}
R(t_1,\ldots,t_r)^{(n-r)/2}
dt_{r} \cdots dt_{2} dt_{1},
\end{align*}
where
$R(t_1,\ldots,t_r)$ is given in Remark~\ref{remark_r_R_R}.
In particular, if $\Hf$ is a Coxeter arrangement, let $0 \leq r \leq n$
be the integer such that $P=P_r$. Then the ball $\Ball(0,R)$ is
sufficiently symmetric with respect to $\Hf$, and
for every $R \geq 0$ and every $a \in V$ such that $\|a\| \leq R$, we have
\begin{align*}
P(\Hf,\Ball(a,R))
& =
\frac{2^r\pi^{r/2}}{\Gamma(r/2+1)}
\sum_{(e_{1},\ldots,e_{r})\in E_{r}^{(0)}}(-1)^{e_{r} \circdots e_{1}}
(a,e_{1})(a,e_{2}) \cdots (a,e_{r})\\
&
\cdot
\int_{0}^{1} \int_{0}^{t_1} \cdots \int_{0}^{t_{r-1}}
(R^2- t_1^2(a,e_{1})^2-t_2^2(a,e_{2})^2-\cdots-t_r^2(a,e_{r})^2)^{(n-r)/2}
dt_{r} \cdots dt_{2} dt_{1}.
\end{align*}
\label{corollary_pizza_ball_parity}
\end{corollary}

\begin{proposition}
Suppose that $\Hf = (H_{e})_{e \in E}$ is an even Coxeter arrangement.
Let $n=\dim(V)$.
Then the polynomial~$f_{\Hf}$ of Theorem~\ref{theorem_pizza_is_polynomial} is given by
\begin{align*}
f_{\Hf}(a)
& =
\begin{cases}
2^{n} \cdot \prod_{e\in E} \bl{a}{e} &
\text{if $\Hf$ is of type $A_{1}^{n}$,} \\
0 & 
\text{otherwise.}
\end{cases}
\end{align*}
More generally, if $\Hf$ is an even arrangement that has a Coxeter subarrangement
$\Hf'$ whose Coxeter group preserves $\Hf$, then
\begin{itemize}
\item[(i)]
if $\Hf'$ has at least $n+1$ hyperplanes then $f_{\Hf}=0$;
\item[(ii)]
if $\Hf'$ has $n$ hyperplanes then $f_{\Hf}$ is a scalar
multiple of the function $a \longmapsto \prod_{e\in E'} \bl{a}{e}$, where
$E'$ is the subset of $E$ corresponding to $\Hf'$.
\end{itemize}
\label{proposition_polynomial_Coxeter_arrangement}
\end{proposition}
\begin{proof}
All even Coxeter arrangements in $V$ have at least $n+1$ hyperplanes,
except for the arrangement of type $A_{1}^{n}$, which has $n$ hyperplanes. 
So it suffices to prove statements (i) and (ii), and to calculate the leading
coefficient of $f_{\Hf}$ for $\Hf$ of type $A_{1}^{n}$.

Suppose that $f_{\Hf} \neq 0$.
Let $X$ be the hypersurface $\{f_{\Hf}=0\}$.
Then $X$ is a hypersurface of degree~$n$, and it contains all the
hyperplanes of $\Hf'$, 
so the arrangement $\Hf'$ has at most $n$ hyperplanes.
Suppose that $\Hf'$ has exactly $n$ hyperplanes. Then
$X=\bigcup_{e\in E'}H_{e'}$, thus the polynomial
$f_{\Hf}$ is of the form $a \longmapsto c \cdot \prod_{e\in E} \bl{a}{e}$
where $c$ is a constant. Now suppose that $\Hf=\Hf'$ is of type~$A_{1}^{n}$.
To calculate~$c$, we use Theorem~\ref{theorem_pizza_is_polynomial}.
Let $K=[-1,1]^{n}$ and $a=(a_{1},\ldots,a_{n})=(1,\ldots,1)$.
Then the cube $K+a$ is entirely in the first orthant
and hence $c=P(\Hf,K+a)=\Vol(K)=2^{n}$.
\end{proof}

\begin{theorem}
Suppose that $\Hf$ is an even Coxeter arrangement in $\Rrr^{n}$
and that $K$ is a measurable subset of $V$ of finite volume and
stable by the Coxeter group of $\Hf$. 
Let $a=(a_{1},\ldots,a_{n})\in\Rrr^{n}$ 
such that $K$ contains the convex hull of the set
$\{w(a) : w\in W\}$.
Then the pizza quantity $P(\Hf,K+a)$ is given by
\begin{align*}
P(\Hf,K+a)
& =
\begin{cases}
2^{n} \cdot a_{1} \cdots a_{n} &
\text{if $\Hf$ is of type $A_{1}^{n}$,} \\
0 &
\text{otherwise.}
\end{cases}
\end{align*}
Here if $\Hf$ has type $A_{1}^{n}$
we assume that it is
given by the hyperplanes $\{x_i=0 : 1\leq i\leq n\}$
and that the base chamber is $T_{0}=(\Rrr_{>0})^{n}$.
\label{theorem_amazing}
\end{theorem}

\begin{proof}
The set $K$ satisfies the conditions of
Theorem~\ref{theorem_pizza_is_polynomial}, so the result
follows from Proposition~\ref{proposition_polynomial_Coxeter_arrangement}.

\end{proof}

\begin{example}
{\rm
Suppose that $|E|=n$ and that $\Hf$ is essential. This implies
that $\Hf$ is even and that for every $0 \leq r \leq n$
the set $P_r$ is the set of $(e_{1},\ldots,e_{r})\in E^r$ such that
$e_i\not=e_j$ for $i\not=j$.
\footnote{More generally, these two statements hold for
any simple hyperplane arrangement.}
Write $E=\{e_{1},\ldots,e_{n}\}$. Then for any $a\in V$ we have
\begin{align*}
f_\Hf(a)
& =
\frac{2^n}{n!}
\sum_{\sigma\in\mathfrak{S}_n}
(-1)^{e_{\sigma(n)} \circdots e_{\sigma(1)}} \\
&
\hspace*{10 mm}
\cdot
(a,e_{\sigma(1)})(\pi_{e_{\sigma(1)}}(a),e_{\sigma(2)})
(\pi_{e_{\sigma(1)},e_{\sigma(2)}}(a),e_{\sigma(3)})
\cdots
(\pi_{e_{\sigma(1)},\ldots,e_{\sigma(n-1)}}(a),e_{\sigma(n)}).
\end{align*}
}
\end{example}

For arrangements that are not Coxeter, we do not know a general way
to find sufficiently
symmetric sets, or even to decide whether
the ball~$\mathbb{B}(0,1)$ is sufficiently symmetric.
We end this section with two low-dimensional examples.

\begin{example}
{\rm
Suppose that $V=\Rrr^{2}$ and $\Hf$ is a line arrangement. Then
$\Hf$ is even if and only if it is nonempty and has an even number of lines, say $2m$
where $m$ is a positive integer.
Suppose that this is the case, and that one of the lines in $\Hf$ is
the horizontal axis $\{x_{2}=0\}$. Let $\theta_{1},\ldots,\theta_{2m-1}$
be the angles between the other lines and the horizontal axis, ordered
so that
$0 = \theta_{0} < \theta_{1} < \theta_{2} < \cdots < \theta_{2m-1} < \theta_{2m} = \pi$.
If $\ell \in \Hf$ and $a \in \ell^{\perp}$ then $\ell \cap (\mathbb{B}(0,1)+a)$ is
always a segment centered at the origin (or empty).
Thus the disc $\mathbb{B}(0,1)$ is sufficienty symmetric with respect to~$\Hf$ if
and only if $P(\Hf,\mathbb{B}(0,1))=0$. This is equivalent
to the condition that $\sum_{i=1}^{m}(\theta_{2i}-\theta_{2i-1})=
\sum_{i=1}^m(\theta_{2i-1}-\theta_{2i-2})$,
that is,
$\sum_{j=1}^{2m} (-1)^{j} \cdot \theta_{j} = \pi/2$.
If this condition is satisfied then the pizza quantity
$P(\Hf,\mathbb{B}(a,R))$ is given by $P(\Hf,\mathbb{B}(a,R))=
f_{\Hf,T_{0}}(a)$ if $0 \in \mathbb{B}(a,R)$, that is, when $\|a\| \leq R$, and
in particular it is independent of the radius $R$ of the ball.
}
\label{example_2d_arrangement}
\end{example}

\begin{example}
{\rm
In $V=\Rrr^3$ consider the arrangement $\Hf$ given by the following
seven hyperplanes:
\begin{align*}
H_{1} & = \{x_{1}=0\}, &
H_{2} & = \{x_{1}=\alpha \cdot x_{2}\}, &
H_{4} & = \{x_{2}=0\}, &
H_{5} & = \{x_{3}=\beta \cdot x_{2}\}, &
H_{7} & = \{x_{3}=0\}, \\
&&
H_{3} & = \{x_{1} = - \alpha \cdot x_{2}\}, &
&&
H_{6} & = \{x_{3} = - \beta \cdot x_{2}\}, 
\end{align*}
where $\alpha$ and $\beta$ are positive real numbers.
We also fix a base chamber $T_{0}$ of $\Hf$.
We claim that this arrangement $\Hf$ is even, and that
the unit ball $\mathbb{B}(0,1)$ is sufficiently symmetric with respect to~$\Hf$.

First, for every chamber $T$ of $\Hf$, the set $-T$ is also a chamber and
$(-1)^{-T}=-(-1)^{T}$ because $\Hf$ has an odd number of hyperplanes.
For any centrally symmetric measurable set of finite volume~$K$,
and in particular for the ball,
we have $P(\Hf,K)=0$. We can then apply Lemma~\ref{lemma_sufficiently_symmetric} 
to check the second statement once we know that $\Hf$ is even.
Note also that for every $H\in\Hf$ and for every $a\in H^\perp$, the intersection
$H\cap(\mathbb{B}(0,1)+a)$ is a disk centered at the origin in $H$.

For every $1 \leq i \leq 7$, let $\Hf_i$ be the even restricted
arrangement induced by $\Hf$ on $H_i$. We need to check that for every $i$
the arrangement $\Hf_i$ is even and a disk centered at the origin is
sufficiently symmetric with respect to $\Hf_i$.
\begin{itemize}
\item[--]
The three arrangements $\Hf_{1}$, $\Hf_{2}$ and $\Hf_{3}$
are all isometric to a four line arrangement, as in
Example~\ref{example_2d_arrangement}.
However, the adjacent angles for $\Hf_{1}$ are given by
$\eta_{1} = \eta_{4} = \arctan(\beta)$
and
$\eta_{2} = \eta_{3} = \pi/2 - \arctan(\beta)$,
where $\theta_{i} = \eta_{1} + \cdots + \eta_{i}$,
whereas, the angles for $\Hf_{2}$ and $\Hf_{3}$ are
$\eta_{1} = \eta_{4} = \arctan(\beta/\sqrt{\alpha^{2}+1})$
and
$\eta_{2} = \eta_{3} = \pi/2 - \arctan(\beta/\sqrt{\alpha^{2}+1})$.

\item[--]
The cases of the arrangements $\Hf_{5}$, $\Hf_{6}$ and $\Hf_{7}$
are symmetric to $\Hf_{3}$, $\Hf_{2}$ and $\Hf_{1}$.

\item[--]
The arrangement $\Hf_{4}$ has two lines and has type $A_1^2$.
\end{itemize}
We conclude that the pizza quantity $P(\Hf,\mathbb{B}(a,R))$ is
given by $P(\Hf,\mathbb{B}(a,R))=f_{\Hf,T_{0}}(a)$, and hence it is independent of
the radius $R$ of the ball.
Furthermore,
Proposition~\ref{proposition_polynomial_Coxeter_arrangement} (ii)
implies that
$f_{\Hf,T_{0}}(a) = c \cdot a_{1} a_{2} a_{3}$
where $c$ is a constant.
Note that we could also get this result directly from
Theorem~\ref{theorem_almost_all_the_products}.
}
\label{example_3d_arrangement}
\end{example}

\section{The case of the ball and convex bodies bounded by quadratic surfaces}
\label{section_ball}

We now revisit the case when the measurable set $K$
is an $n$-dimensional ball.
We show in particular that the pizza quantity vanishes for
the arrangements $E_{6}$ and 
$A_{n}$ where $n \equiv 0,1 \bmod 4$.
Recall that 
$V$ is an $n$-dimensional vector space endowed with an inner product
$(\cdot,\cdot)$ and that $\Hf = (H_{e})_{e \in E}$ is an arrangement
with base chamber~$T_{0}$.

\begin{definition}
We say that the hyperplane arrangement $\Hf$
satisfies \emph{the parity condition} if
$|\Hf|$ and $\dim(V)$ have the same parity.
\end{definition}

\begin{lemma}
\begin{itemize}
\item[(i)] Suppose that $\Hf$ is not the empty arrangement.
Then the following three statements are equivalent:
\begin{itemize}
\item[(a)] The arrangement $\Hf$ satisfies the parity condition.
\item[(b)] There exists $e\in E$ such that the
even restricted arrangement~$\Hf_{e}$
satisfies the parity condition.
\item[(c)] For every $e\in E$, the even restricted arrangement~$\Hf_{e}$
satisfies the parity condition.

\end{itemize}

\item[(ii)] The following statements are equivalent:
\begin{itemize}
\item[(a)] The arrangement $\Hf$ satisfies the parity condition.
\item[(b)] 
There exists $(e_{1},\ldots,e_{r})\in P$ such that
$\dim(H_{e_{1},\ldots,e_{r}})$ is even.
\item[(c)] 
For every $(e_{1},\ldots,e_{r})\in P$, we have that
$\dim(H_{e_{1},\ldots,e_{r}})$ is even.

\end{itemize}

\item[(iii)] If $\Hf$ is an even arrangement then it satisfies the parity condition.

\item[(iv)] Suppose that $\Hf$ is a Coxeter arrangement. We write
$V = V_{0} \oplus V_{1}$, where $V_{0} = \bigcap_{H\in\Hf} H$
and $V_{1} = V_{0}^\perp$, and
we denote by $\Hf_{1}$ the restriction of $\Hf$ to $V_{1}$, that is, the
arrangement $\{H \cap V_{1} : H \in \Hf\}$. Note that $\Hf_{1}$ is 
an essential Coxeter arrangement and
that $\Hf$ is the Cartesian product of $\Hf_{1}$ and of the empty arrangement
on $V_{0}$. We can also decompose $\Hf_{1}$ as a product of simple Coxeter
arrangements. Let $r$ be the number of arrangements in this decomposition
that are of type $A_{n}$ with $n \equiv 2,3 \bmod 4$, $D_{n}$ for $n\geq 5$ odd or
$I_{2}(2k+1)$ for $k \geq 2$. Then $\Hf$ satisfies the parity condition 
if and only if $r+\dim(V_{0})$ is even.

In particular, if $\Hf$ is an essential Coxeter arrangement then
it satisfies the parity condition if and only it has an even number of
simple factors
of types $A_{n}$ with $n \equiv 2, 3 \bmod 4$,
$D_{n}$ for $n\geq 5$ odd or
$I_{2}(2k+1)$ for $k\geq 2$.
\end{itemize}
\label{lemma_parity_condition}
\end{lemma}
\begin{proof}
We begin by proving~(i).
As $E$ is nonempty, condition~(c) clearly implies condition~(b). 
Let $e\in E$. 
We denote by $\sim$ the equivalence relation on $E-\{e\}$ defined
by $e'\sim e''$ if and only if $H_{e} \cap H_{e'}=H_{e} \cap H_{e''}$.
Let $E_{1}$, respectively $E_{2}$,
be the set of $e'\in E-\{e\}$ such that the intersection multiplicity
of $H_{e} \cap H_{e'}$ is odd, respectively even. Then $E_{1}$ and $E_{2}$ are unions
of equivalence classes of the relation~$\sim$.
Since the element $e$ is not in the set $E-\{e\}$, the intersection multiplicity of $H_{e} \cap H_{e'}$
has the opposite parity of the cardinality of the equivalence class $e'$ belongs to.
As all the equivalence classes contained in~$E_{1}$ have even cardinality, the set~$E_{1}$ also has
even cardinality. Moreover, we have a surjective map from~$E_{2}$
to $\Hf_{e}$ sending $e'\in E_{2}$ to $H_{e}\cap H_{e'}$, and the fibers of this map
are the equivalence classes contained in $E_{2}$, all of which have
odd cardinality. Thus we obtain $|E_{2}| \equiv |\Hf_{e}| \bmod 2$. We deduce that
$|\Hf| = |E|=1+|E_{1}|+|E_{2}| \equiv 1+|\Hf_{e}| \bmod 2$. This shows that $\Hf$ satisfies
the parity condition if and only if $\Hf_{e}$ does. As $e$ was arbitrary,
this argument proves that (a) implies~(c) and that (b) implies~(a).

Next we prove~(ii) by induction on $|\Hf|$. Note that (c) always
implies (b), so we just need to prove that (a) implies (c) and
that (b) implies (a).
If $\Hf$ is empty then $P=\varnothing$, and the three conditions
(a), (b) and (c) are equivalent to the fact that $\dim(V)$ is even.
Suppose that $|\Hf|\geq 1$.
If $\Hf$ satisfies the
parity condition, let $(e_{1},\ldots,e_{r})\in P$. We have $r\geq 1$
because $\Hf$ is not empty, so $\Hf_{e_{1}}$ satisfies the parity
condition by~(ii). As $(e_{2},\ldots,e_{r})$ is a maximal even restriction
sequence for~$\Hf_{e_{1}}$, the induction hypothesis implies
that $\dim(H_{e_{1},\ldots,e_{r}})$ is even. Conversely, suppose that
there exists $(e_{1},\ldots,e_{r})\in P$ such that $\dim(H_{e_{1},\ldots,e_{r}})$
is even.
By the induction hypothesis
$\Hf_{e_{1}}$
satisfies the parity condition
so $\Hf$ satisfies the parity condition by (ii).

If $\Hf$ is even then $P=P_n$, so $H_{e_{1},\ldots,e_{n}}=\{0\}$ for every
$(e_{1},\ldots,e_{n})\in P$. Hence (iii) follows from~(ii).

Finally, we prove~(iv).
Simple Coxeter arrangements satisfy the parity condition
if and only if they are of types $A_{n}$ with $n \equiv 0,1 \bmod 4$, $B_{n}$,
$D_{n}$ with $n \geq 4$ even,
$E_{6}$, $E_{7}$, $E_{8}$, $F_{4}$,
or $I_{2}(2k)$ with $k\geq 2$.
Thus $|\Hf| + \dim(V) \equiv r + \dim(V_{0}) \bmod 2$, which implies the result.
\end{proof}

\begin{remark}
{\rm
By
Lemma~\ref{lemma_parity_condition}
the evenness condition implies the parity condition,
but it is much stronger condition.
(In dimensions at most~$2$, the two conditions are equivalent).
For example,
the Coxeter arrangements of type $A_{2} \times A_{3}$,
$E_{6}$ and $A_{n}$ for $n \equiv 0,1 \bmod 4$ satisfy the
parity condition but not the evenness condition.
}
\label{remark_parity_condition}
\end{remark}

If $\Hf$ satisfies the parity condition then $(n-r)/2$ is an integer
for every $(e_1,\ldots,e_r)\in P$, so if the ball $\Ball(0,R)$ is
sufficiently symmetric with respect to $\Hf$, 
Corollary~\ref{corollary_pizza_ball_parity}
implies that the pizza
quantity $P(\Hf,\Ball(a,R))$ is polynomial in $a$ and $R$ as long
as $\|a\|\leq R$.

In this section we will see that we do not need the condition
that the ball is sufficiently symmetric for this result.
We will actually consider slightly more general ``balls'' that are
convex bodies bounded by quadratic hypersurfaces. Let
$q$ be a positive definite quadratic form on $V$, so that
$(q(x+y)-q(x)-q(y))/2$ is an inner product on $V$.
For $a\in V$ and $R \geq 0$, we write
\[\Ball_q(a,R)=\{x\in V : q(x-a)\leq R^2\}.\]
If $q$ is the quadratic form defined by $q(x)=(x,x)$ then
$\Ball_q(a,R)$ is the usual ball $\Ball(a,R)$ with center $a$ and
radius $R$. Note also that $\Ball_q(a,R)=a+\Ball_q(0,R)$ and
$\Ball_q(0,R)=R\cdot\Ball_q(0,1)$.

\begin{theorem}
Let $C$ be a closed convex polyhedral cone, that is, an intersection
of closed half-spaces in $V$.
Then there exists a polynomial function
$g_{C,q}$ on $\Rrr\times V$ such that:
\begin{itemize}
\item[(a)]
The polynomial $g_{C,q}$ is homogeneous of degree $n=\dim(V)$.
\item[(b)]
For every $(R,a)\in\Rrr\times V$, we have $g_{C,q}(R,a)=g_{C,q}(-R,a)$,
that is, $g_{C,q}$ 
contains only terms of even degree in the first variable $R$.
\item[(c)]
For every $(R,a)\in\Rrr\times V$ such that $q(a)\leq R^2$, we have
\begin{align*}
\Vol(C\cap\ \Ball_q(a,R))+(-1)^{\dim(V)}\Vol((-C)\cap\ \Ball_q(a,R))
& =
g_{C,q}(R,a) .
\end{align*}
\end{itemize}
\label{theorem_more_polynomial_pizza_by_sector}
\end{theorem}
\begin{proof}
We prove the result by induction on the dimension of $V$.
If $\dim(V)=0$, then $V=C=\{0\}$ and
$\Vol(C\cap\ \Ball_q(a,R))+(-1)^{\dim(V)}\Vol((-C)\cap\ \Ball_q(a,R))=2$
if $q(a)\leq R^2$, so we can take $g_{C,q}=2$.

Suppose that $\dim(V)\geq 1$ and that we know the result for
lower-dimensional inner product spaces.
Let $(R,a)\in\Rrr\times V$ such that $q(a)\leq R^2$.
Let $\Hf$ be the set of hyperplanes containing a facet of $C$.
For every $H\in\Hf$, we denote by $e_H$ the unit normal vector of
$H$ that points to the half-space containing $C$, and we set
$E=\{e_H : H\in\Hf\}$.
By Theorem~\ref{theorem_reduction}(i) we have 
\begin{align*}
&
\Vol(C\cap\ \Ball_q(a,R))+(-1)^{\dim(V)}\Vol((-C)\cap\ \Ball_q(a,R)) \\
&
-(\Vol(C\cap\ \Ball_q(0,R))+(-1)^{\dim(V)}\Vol((-C)\cap\ \Ball_q(0,R)))
\\
= \: &2 \cdot \sum_{e\in H}
\bl{a}{e}
\cdot
\int_{0}^{1}\left(\Vol(C\cap H_e\cap\ \Ball_q(sa,R))
+
(-1)^{\dim(H_e)}
\Vol((-C)\cap H_e\cap\ \Ball_q(sa,R))\right)\:ds.
\end{align*}
Indeed, the facets of $C$ and $-C$
are exactly the relative interiors of the intersections
$C\cap H_e$, respectively $(-C)\cap H_e$, for $e\in E$.
Moreover, for every $e\in E$,
as the vector $e$ points towards~$C$, hence away from $-C$,
we have $(-1)^{\mathring{C}}(-1)^{U\circ e}=1$ 
if $U$ is the relative interior of $C\cap H_e$ and
$(-1)^{\mathring{C}}(-1)^{U\circ e}=-1$ 
if $U$ is the relative interior of $(-C)\cap H_e$.

Let $e \in E$.
We denote by $\pi_{q,e}$ the orthogonal projection on $H_{e}$
with respect to the inner product $(q(x+y)-q(x)-q(y))/2$
corresponding to the quadratic form $q$. 
Then for all
$a\in V$ and $x\in H_e$, we have
\[q(x-a)=q(x-\pi_{q,e}(a))+q(a-\pi_{q,e}(a)).\]
In particular, the convex body $\Ball_q(sa,R)\cap H_{e}$
is given by
\begin{align*}
\Ball_q(sa,R) \cap H_{e}
& =
\Ball_{q_e}\left(s\pi_{q,e}(a),\sqrt{R^2-s^2q(a-\pi_{q,e}(a))}\right) ,
\end{align*}
where $q_e$ is the restriction of $q$ to $H_e$.
Hence the induction hypothesis applied to $H_e$ yields that
\begin{align*}
&
\Vol(C\cap H_e\cap\ \Ball_q(sa,R))+
(-1)^{\dim(H_e)}\Vol((-C)\cap H_e\cap\ \Ball_q(sa,R)) \\
= \:&
g_{C\cap H_e,q_e}\left(\sqrt{R^{2} - s^{2} \cdot q(a-\pi_{q,e}(a))}, s\pi_{q,e}(a)\right) .
\end{align*}
By conditions~(a) and~(b), there exist homogeneous polynomial functions $g_{C\cap H_e,i}$
of degree $n-1-2i$ on $H_e$ such that, for every $(t,x)\in\Rrr\times H_e$,
we have
$g_{C\cap H_e}(t,x) = \sum_{i=0}^{\lfloor (n-1)/2\rfloor}t^{2i} \cdot g_{C\cap H_e,i}(x)$.
We obtain
\begin{align*}
&
\int_{0}^{1}\left(\Vol(C\cap H_e\cap\ \Ball_q(sa,R))+
(-1)^{\dim(H_e)}
\Vol((-C)\cap H_e\cap\ \Ball_q(sa,R))\right)
\: ds \\
= \: &
\sum_{i=0}^{\lfloor \frac{n-1}{2} \rfloor}
g_{C\cap H_e,i}(\pi_{q,e}(a))
\cdot
\int_{0}^{1}
\left(R^{2} - s^{2} \cdot q(a-\pi_{q,e}(a))\right)^{i} 
\cdot s^{n-1-2i} \: ds .
\end{align*}
We define a function $g^{(0)}_{C}:\Rrr\times V \longrightarrow\Rrr$ by
\begin{align*}
g^{(0)}_{C}(t,x)
& =
2 \cdot \sum_{e\in E}
\bl{x}{e} 
\cdot \sum_{i=0}^{\lfloor \frac{n-1}{2} \rfloor}
g_{\Hf_{e},i}(\pi_{q,e}(x))
\cdot
\int_{0}^{1}\left(t^{2} - s^{2} \cdot q(x-\pi_{q,e}(x))\right)^{i}
\cdot s^{n-1-2i} \: ds ,
\end{align*}
for $(t,x) \in \Rrr \times V$. As the functions $x\longmapsto (x,e)$ and
$x \longmapsto \pi_{q,e}(x)$ are linear in $x$ and the function
$x\longmapsto q(x)$ is quadratic in $x$,
the function~$g_C^{(0)}$ is
a homogeneous polynomial of degree $\dim(V)$ that satisfies condition (b), and
we have 
\begin{align*}
&
\Vol(C\cap\ \Ball_q(a,R))+(-1)^{\dim(V)} \Vol((-C)\cap\ \Ball_q(a,R)) \\
= & 
\Vol(C\cap\ \Ball_q(0,R))
+(-1)^{\dim(V)}\Vol((-C)\cap\ \Ball_q(0,R))
+g_C^{(0)}(R,a).
\end{align*}
Suppose that $n=\dim(V)$ is even.
As $\Vol(Z\cap\ \Ball_q(0,R)) = R^{n} \cdot \Vol(Z\cap\ \Ball_q(0,1))$
for $Z=C$ or $Z=-C$,
the function $g_{C,q}$ defined by
\begin{align*}
g_{C,q}(t,x)
& =
t^{n} \cdot  (\Vol(C\cap\ \Ball_q(0,1))+\Vol((-C)\cap\Ball_q(0,1))) + g_C^{(0)}(t,x)
\end{align*}
satisfies the desired properties. Suppose that $\dim(V)$ is odd.
Then, as $\Ball_q(0,R)$ is centrally symmetric, we have
$\Vol(C\cap\ \Ball_q(0,R))-\Vol((-C)\cap\Ball_q(0,R))=0$ for
every $R\geq 0$,
so the function
$g_{C,q} = g_{C}^{(0)}$ satisfies the desired properties.
\end{proof}

\begin{corollary}
Suppose that $\Hf$ satisfies the parity condition. 
Then there exists a polynomial function
$g_{\Hf,q}$ on $\Rrr\times V$ such that:
\begin{itemize}
\item[(a)]
The polynomial $g_{\Hf,q}$ is homogeneous of degree $n=\dim(V)$.
\item[(b)]
For every $(R,a)\in\Rrr\times V$, we have $g_{\Hf,q}(R,a)=g_{\Hf,q}(-R,a)$,
that is, $g_{\Hf,q}$ 
contains only terms of even degree in the first variable $R$.
\item[(c)]
For every $(R,a)\in\Rrr\times V$ such that $q(a)\leq R^2$, we
have $P(\Hf,\Ball_q(a,R))=g_{\Hf,q}(R,a)$.
\end{itemize}
\label{corollary_more_polynomial_pizza}
\end{corollary}
\begin{proof}
As $\Hf$ satisfies the parity condition, we have
$(-1)^{-T}=(-1)^{\dim(V)}(-1)^T$ for every chamber~$T$ of $\Hf$,
hence $P(\Hf,\Ball_q(a,R))$ is an alternating sum of quantities
as in Theorem~\ref{theorem_more_polynomial_pizza_by_sector}.
More precisely, choose a subset $\Tf'$ of $\Tf(\Hf)$ such that,
for every chamber $T$ of $\Hf$, exactly one of $T$ or $-T$ is in
$\Tf'$. Then, for all $a\in V$ and $R\in\R_{\geq 0}$, we have
\[P(\Hf,\Ball_q(a,R))=\sum_{T\in\Tf'}(-1)^T\left(
\Vol(T\cap\ \Ball_q(a,R))+(-1)^{\dim(V)}\Vol((-T)\cap\ \Ball_q(a,R))
\right).\]
So it suffices to take
\begin{align*}
g_{\Hf,q} & = \sum_{T\in\Tf'}(-1)^T g_{\overline{T},q}.
\qedhere
\end{align*}
\end{proof}

\begin{theorem}
Let $q$ be a positive definite quadratic form on $V$.
Suppose that $\Hf$ satisfies the parity condition and that it
contains a Coxeter arrangement $\Hf'$ whose Coxeter group preserves
$\Hf$ and the convex body $\Ball_q(0,1)$.
(For example, these conditions are satisfied if $\Hf=\Hf'$ is
a Coxeter arrangement and $q$ is given by $q(x)=(x,x)$.)
Let $(R,a)\in\Rrr\times V$ such that $q(a)\leq R^2$,
equivalently $0 \in \Ball_q(a,R)$.
\begin{itemize}
\item[(i)] If $|\Hf'|>\dim(V)$, so in particular $\Hf'$ is not of type $A_{1}^{\dim(V)}$,
then the pizza quantity
$P(\Hf,\Ball_q(a,R))$ vanishes.

\item[(ii)] If $|\Hf'|=\dim(V)$ (for example if $\Hf'$ is of type
$A_{1}^{\dim(V)}$, or the product of a type $A_{2}$ arrangement
on~$\Rrr^{2}$ and the empty
arrangement on $\Rrr$), 
then there exists a constant $c$ independent of the center~$a$ and the radius~$R$ such
that $P(\Hf,\Ball_q(a,R))=
c\cdot\prod_{e\in E'}(a,e)$, where 
$E'$ is the subset of $E$ corresponding to the hyperplanes of $\Hf'$.
In particular, 
the pizza quantity
$P(\Hf,\Ball_q(a,R))$ is independent of the radius~$R$.
\end{itemize}
\label{theorem_almost_all_the_products}
\end{theorem}
\begin{proof}
The proof is very similar to that of 
Proposition~\ref{proposition_polynomial_Coxeter_arrangement}.
Let $W$ be the Coxeter group of $\Hf'$.
Fix $R > 0$, and define a function
$h:V \longrightarrow \Rrr$ by $h(x)=g_{\Hf,q}(R,x)$. 
Then $h$ is polynomial of degree at most $\dim(V)$,
and we have $P(\Hf,\Ball_q(a,R))=h(a)$ for every $a\in V$ such that
$q(a)\leq R^2$. By Corollary~\ref{corollary_Coxeter_on_a_hyperplane} this
implies that $h(w(x))=(-1)^{w} \cdot h(x)$ for every $w\in W$ and $x\in V$,
and in particular that $h(x)=0$ if $x$ is on one of the hyperplanes of
$\Hf'$. If $h$ is nonzero then the hypersurface $\{h=0\}$ is
of degree at most $\dim(V)$ and contains $\bigcup_{H\in\Hf'}H$.
But this is impossible
in the situation of (i), so we conclude that $h=0$ in that situation.
In the situation of (ii), this
is possible and implies that $h$ is of the form
$x \longmapsto c(R) \cdot \prod_{e\in E'} \bl{x}{e}$,
for some $c(R)\in\Rrr$ not depending on $x$. 
By definition of $h$ we obtain that
$g_{\Hf,q}(R,a)=c(R) \cdot \prod_{e\in E'} \bl{a}{e}$
for every $(R,a)\in\Rrr\times V$
such that $q(a)\leq R^2$. As $g_{\Hf,q}$ and the function
$x \longmapsto \prod_{e\in E'} \bl{x}{e}$
are both polynomials homogeneous of degree $\dim(V) = |\Hf'|$,
this implies that the function $R \longmapsto c(R)$ is constant,
completing the proof.
\end{proof}

\begin{remark}
{\rm
If $\Hf$ is the product of the empty arrangement in $\Rrr$ and a type
$I_{2}(2k+1)$ arrangement in $\Rrr^{2}$ with $k\geq 2$, then the corollary
recovers the case of Entr\'ee~2 on page~433
of~\cite{Mabry_Deiermann} where $N\geq 5$ is odd. 
}
\end{remark}

If the arrangement $\Hf$ does not satisfy the parity condition,
then it is not true in general
that the pizza quantity $P(\Hf,\mathbb{B}(a,R))$ is polynomial in $a$ and $R$.
See for example the calculations on page~429 of the paper~\cite{Mabry_Deiermann}.
We can, however, control its behavior as $R$ tends to infinity in the case where
$\Hf$ is a Coxeter arrangement.

\begin{proposition}
Suppose that $\Hf$ is a Coxeter arrangement with Coxeter group $W$
in an $n$-dimensional vector space $V$.
Let $K$ be a measurable subset of $V$ with finite volume that is stable by
$W$ and contains a neighborhood of $0$. Suppose that there exists
an integer $k\geq 0$ and a function $h:V \longrightarrow \Rrr$
of type $C^{k}$ at $0$ such
that $P(\Hf,K+a)=h(a)$ for $a$ in a neighborhood of $0$. 

Fix the center $a\in V$. Then $0\in a+rK$ for $r>0$ large enough and the function
$P(\Hf,a+rK)$ satisfies
\begin{align*}
P(\Hf,a+rK) &
=
o\left(r^{n-\min(k,|\Hf|-1)}\right) 
\end{align*}
as $r \longrightarrow +\infty$. 
In particular, if $k\geq n$ and $|\Hf|\geq n+1$ then we have
${\displaystyle \lim_{r \to +\infty}P(\Hf,a+rK)=0}$.
\label{proposition_going_to_infinity}
\end{proposition}
\begin{proof}
By the multivariate version of Taylor's theorem, there exist polynomial
functions $h^{(i)}:V\longrightarrow\Rrr$ for $0\leq i\leq k$ such that
$h^{(i)}$ is homogeneous of degree $i$ and
\begin{align*}
h(a)
& =
\sum_{i=0}^k \frac{1}{i!} \cdot h^{(i)}(a) + o\left(\|a\|^{k}\right).
\end{align*}
We also have explicit formulas for the $h^{(i)}$ involving partial
differentials of $h$. In particular the fact that
$h(w(b))=(-1)^{w} \cdot  h(b)$ for $w\in W$ and $b$ in a neighborhood of
$0$ implies that $h^{(i)}(w(a))=(-1)^{w} \cdot  h^{(i)}(a)$ for every
$i$, every $w\in W$ and every $a\in V$. Let $0 \leq i \leq k$, and
suppose that $i\leq|\Hf|-1$. As in the proof of 
Proposition~\ref{proposition_polynomial_Coxeter_arrangement}, if $h^{(i)} \neq 0$,
then the hypersurface $\{h^{(i)}=0\}$ is of degree $i$ and contains
$\bigcup_{H\in\Hf} H$. This is not possible hence $h^{(i)} = 0$.
If $k \geq |\Hf|$, we deduce that
$h(a) = h^{(|\Hf|)}(a) + o(\|a\|^{|\Hf|}) = o(\|a\|^{|\Hf|-1})$.
If $k \leq |\Hf|-1$, we deduce that $h(a) = o(\|a\|^{k})$.
So we conclude that $h(a)=o(\|a\|^{\min(k,|\Hf|-1)})$.

We now fix $a\in V$. If $r>0$ then $a+rK = r \cdot (a/r+K)$, so $0\in a+rK$ as soon
as $-a/r\in K$. This holds for large enough $r$ by the assumption on $K$.
Also, we have
$P(\Hf,a+rK) = r^{n} \cdot P(\Hf,a/r+K) = r^{n} \cdot  h(a/r)$ for large enough $r$,
so $P(\Hf,a+rK)=o(r^{n-\min(k,|\Hf|-1)})$. If $|\Hf|\geq n+1$ and $k\geq n$,
this implies in particular that $P(\Hf,a+rK)=o(1)$.
\end{proof}

\begin{remark}
{\rm
As in Theorem~\ref{theorem_almost_all_the_products}, we
just need in Proposition~\ref{proposition_going_to_infinity}
the fact that $\Hf$ contains a Coxeter subarrangement
$\Hf'$ whose Coxeter groups stabilizes both the arrangement $\Hf$ and the set~$K$.
}
\end{remark}

\begin{corollary}
Suppose that $\Hf$ is a Coxeter arrangement
in an $n$-dimensional inner product space~$V$
that has at least $n+1$ hyperplanes.
Let $q$ be a positive definite quadratic form on~$V$, and
fix a point $a\in V$.
Then the function $R \longmapsto P(\Hf,\Ball_q(a,R))$ is
$o(R^{n-|\Hf|+1})$ as $R \longrightarrow +\infty$, and in particular,
${\displaystyle \lim_{R \to +\infty}P(\Hf,\Ball_q(a,R))=0}$.
\label{corollary_going_to_infinity}
\end{corollary}
\begin{proof}
We apply Proposition~\ref{proposition_going_to_infinity}
with $K=\Ball_q(0,1)$. In that case, the function
$a \longmapsto P(\Hf,a+K)$ is $C^{\infty}$ (and even real analytic) in $a$.
\end{proof}

\section{Surface volume}
\label{section_surface_volume}

We now change our discussion from measuring the volume
of the regions $K \cap T$ to measuring the surface volume
in the case when the convex body is a ball.
For an $n$-dimensional convex set~$X$,
let $\Vol_{n-1}(\partial X)$ denote the $(n-1)$-dimensional surface volume of the set~$X$.
\begin{theorem}
Assume that $\Hf$ is an $n$-dimensional hyperplane arrangement
such that the pizza quantity $P(\Hf, \Ball(a,R))$ does not depend on
the radius $R \geq \|a\|$.
Then the alternating sum of the surface volumes
of the regions $\Ball(a,R) \cap \overline{T}$
where $T$ ranges over all chambers of the arrangement~$\Hf$
is zero, that is,
\begin{align}
\sum_{T \in \Tf} (-1)^{T} \cdot \Vol_{n-1}(\partial(\Ball(a,R) \cap \overline{T}))
& = 0 .
\label{equation_surface_volume}
\end{align}
\label{theorem_surface_volume}
\end{theorem}
\begin{proof}
Since the ball $\Ball(a,R)$ grows uniformly in each direction as
$R$ increases, taking the derivative of the pizza quantity
with respect to $R$ yields
\begin{align*}
\sum_{T \in \Tf} (-1)^{T} \cdot \Vol_{n-1}(\Sss(a,R) \cap T)
& = 0 ,
\end{align*}
where $\Sss(a,R)$ denotes the $(n-1)$-dimensional sphere
having center $a$ and radius $R$,
that is, $\Sss(a,R) = \{x \in V : \| x-a \| = R \}$.
The result follows by observing that each subchamber
contributes its $(n-1)$-dimensional volume to
two terms in~\eqref{equation_surface_volume}
with opposite signs.
\end{proof}

Combining Theorem~\ref{theorem_almost_all_the_products}
with Theorem~\ref{theorem_surface_volume}
yields the following result.
\begin{theorem}
Suppose that $\Hf$ satisfies the parity condition and that it
contains a Coxeter arrangement $\Hf'$ whose Coxeter group preserves
$\Hf$ and such that $|\Hf'| \geq \dim(V)$.
Let $R \geq \|a\|$.
Then the alternating sum of the surface volumes
of the regions $\Ball(a,R) \cap \overline{T}$
where $T$ ranges over all chambers of the arrangement~$\Hf$
is zero, that is,
\begin{align*}
\sum_{T \in \Tf} (-1)^{T} \cdot \Vol_{n-1}(\partial(\Ball(a,R) \cap \overline{T}))
& = 0 .
\end{align*}
\label{theorem_surface_volume_two}
\end{theorem}

\begin{remark}
{\rm
If $\Hf$ is the product of the trivial arrangement on $\Rrr$ and
of a type~$A_{2}$ or $I_{2}(2k+1)$ with $k\geq 2$ arrangement
on~$\Rrr^{2}$ then
Theorem~\ref{theorem_surface_volume_two} is 
Confection~2 on page~433 of~\cite{Mabry_Deiermann}. If $\Hf$
is a type $I_{2}(2k)$ for $k\geq 2$ arrangement on $\Rrr^{2}$ then
Theorem~\ref{theorem_surface_volume_two} is the ``$N$ even'' part of
Confection~3 on page~434 of~\cite{Mabry_Deiermann}.
}
\end{remark}

\begin{remark}
{\rm
We can also apply Theorem~\ref{theorem_surface_volume} to all the even
hyperplanes arrangements for which a ball centered at the origin is
sufficiently symmetric. See for example the arrangement of 
Example~\ref{example_2d_arrangement}. Also,
Theorem~\ref{theorem_surface_volume_two}
applies to the arrangement of Example~\ref{example_3d_arrangement}. 
}
\end{remark}

\section{Concluding remarks}
\label{section_concluding_remarks}

The $n$-dimensional volume and the $(n-1)$-dimensional
surface volume are both examples of intrinsic volumes;
see~\cite{Klain_Rota,Schneider}.
In the paper~\cite{Ehrenborg_Morel_Readdy2}, we have generalized
Theorem~\ref{theorem_first_in_introduction} to all intrinsic volumes, and
more generally, to all valuations on convex subsets of $\Rrr^n$ that are
invariant by affine isometries.  The methods we develop are different from the ones
used here. However, it is still an open question whether
Theorems~\ref{theorem_second_in_introduction}
and~\ref{theorem_surface_volume}
can be generalized to all intrinsic volumes.
Naturally, the result is true for the $0$th intrinsic volume,
that is, the Euler characteristic.

Another generalization of the Pizza Theorem
is to consider the problem of sharing pizza
among more than two people.
We can use Theorem~\ref{theorem_more_polynomial_pizza_by_sector}
to produce such pizza-sharing results for the ball.  The general idea is that
if we consider a sum of terms as in that theorem that has enough symmetries,
we will be able to show that it vanishes by the method of
Proposition~\ref{proposition_polynomial_Coxeter_arrangement}
and Theorem~\ref{theorem_almost_all_the_products}.
We state one such result in dimension $2$, where we can eliminate 
the assumption that the pizza is a disc.

\begin{proposition}
Let $\Hf$ be a Coxeter arrangement of type $I_{2}(k)$ in the plane $V$,
and let $p<k$ be a positive integer dividing $k$.
Let $K$ be a measurable subset of $V$ with finite volume
that is stable by the Coxeter group $W$ of $\Hf$.
If $k$ is odd, or if $k=2p$ with $p$ odd, we also suppose that
$K$ is stable by the Coxeter group of the Coxeter arrangement of
type $I_2(2k)$ containing $\Hf$.
Let $T_{0}, T_{1}, \ldots, T_{2k-1}$ be the chambers of $\Hf$.
For $a\in V$ such that $K$ contains the convex hull of
the set $\{w(a) : w\in W\}$,\footnote{This condition holds for example
if $K$ is convex and $0\in K+a$.}
the following sum is independent of $0 \leq r \leq p-1$:
\[
\sum_{i=0}^{2k/p-1} \Vol(T_{r+pi} \cap (K+a))
=
\Vol(K)/p .
\]
\label{proposition_k_people}
\end{proposition}
In short, $p$ people can share a pizza and have $2k/p$ slices each.

\begin{proof}[Proof of Proposition~\ref{proposition_k_people}.]
If $T$ is a chamber of the arrangement $\Hf$
and $K$ is a measurable subset with finite volume, we deduce 
from point~(i) of Theorem~\ref{theorem_reduction} that
\begin{align}
\nonumber
&
\Vol(T\cap(K+a)) + \Vol((-T)\cap(K+a)) - (\Vol(T\cap K) + \Vol(-T\cap K)) \\
= \: &
(a,e) \int_0^1 P(\Hf_e, (K+ta) \cap H_e) \: dt
+
(a,e') \int_0^1 P(\Hf_{e'}, (K+ta) \cap H_{e'}) \: dt,
\label{equation_K_four}
\end{align}
where $e,e'\in E$ are such that the boundaries of $T$ and $-T$ are
contained in $H_e\cup H_{e'}$. If $K$ satisfies the conditions
of the proposition, then for every $e\in E$
the intersection $(K+a-\pi_e(a))\cap H_e$ 
is centrally symmetric in $H_e$ and contains the interval
$[-\pi_e(a),\pi_e(a)]$. 
(To see that $(K+a-\pi_e(a))\cap H_e$ is centrally symmetric, we use
the fact that $K$ is stable under the orthogonal reflection in the line
perpendicular to $H_e$. If $k$ is even, this line is part of the arrangement
$\Hf$, but if $k$ is odd, it is only part of the larger arrangement
of type $I_2(2k)$.  This is why we assume that $K$ is stable under
the Coxeter group of that larger arrangement.)
This implies that there exists a linear
function $g_e:H_{e} \longrightarrow \R$ such that
$P(\Hf_e,(K+a)\cap H_e)=g_e(\pi_e(a))$ 
for all $K$ and $a$ satisfying
the conditions on the proposition.
So we deduce that the right-hand side of~\eqref{equation_K_four}
is a polynomial homogeneous of degree~$2$ in~$a$ that does not
depend on~$K$, as long as $K$ contains the convex hull of the set
$\{w(a) : w\in W\}$.

Recall that the chambers are labelled $T_{0}, T_{1}, \ldots, T_{2k-1}$,
where the index is modulo $2k$.
Note that $-T_i=T_{i+k}$.
Let $\ell_{i}$ be the line that borders the chambers $T_{i}$ and $T_{i+1}$;
hence $\ell_{i}$ also borders $T_{i+k}$ and $T_{i+k+1}$. We take
the index $i$ of the lines to be an integer modulo $k$.
For every $i\in\Z/k\Z$ and every $j\in\Z/2k\Z$, the orthogonal reflection
in the line $\ell_i$ sends $T_j$ to $T_{2i+1-j}$.

For $0\leq r\leq p-1$ define
\[S_r(K+a)=\sum_{i=0}^{2k/p-1}\Vol(T_{r+pi}\cap(K+a)).\]
By equation~\eqref{equation_K_four}, for every $r$,
there exists a homogeneous polynomial $f_r: V \longrightarrow \R$
of degree $2$ such that $S_r(K+a)-S_r(K)=f_r(a)$ 
for all $K$ and $a$ satisfying the conditions in the proposition.
Fix $0 \leq r \leq p-2$.
For every $0 \leq j \leq k/p-1$ the orthogonal reflection in the line $\ell_{r+jp}$ sends
$\bigcup_{i=0}^{2k/p-1}T_{r+pi}$ to
$\bigcup_{i=0}^{2k/p-1}T_{r+1+pi}$, so $S_r(K+a)=
S_{r+1}(K+a)$ if $a\in\ell_{r+jp}$. We deduce that the polynomial
$f_{r}-f_{r+1}$ vanishes on $\bigcup_{j=0}^{k/p-1}\ell_{r+jp}$. If
$k/p\geq 3$ then this union contains at least three lines.
As the polynomial $f_{r}-f_{r+1}$ is homogeneous of degree $2$, it
has to be zero, so $S_r(K+a)-S_r(K)=S_{r+1}(K+a)-S_{r+1}(K)$ for
all $K$ and $a$ satisfying the condition of the proposition. As
$S_r(K)=S_{r+1}(K)$ because $K$ is stable by $W$, we finally obtain
that $S_r(K+a)=S_{r+1}(K+a)$.

We finally consider the case where $k/p=2$, that is, $k=2p$.
If $p$ is even then the orthogonal reflection in the line
$\ell_{r+p/2}$ also send 
$\bigcup_{i=0}^{2k/p-1}T_{r+pi}$ to
$\bigcup_{i=0}^{2k/p-1}T_{r+1+pi}$, so the polynomial 
$f_r-f_{r+1}$ vanishes on the line $\ell_{r+p/2}$, and we can
again deduce that $f_r-f_{r+1}=0$ and that $S_r(K+a)=S_{r+1}(K+a)$.
If $p$ is odd then the orthogonal reflection in the line $\ell$ bisecting
$\ell_{r+(p-1)/2}$ and $\ell_{r+(p+1)/2}$ sends
$\bigcup_{i=0}^{2k/p-1}T_{r+pi}$ to
$\bigcup_{i=0}^{2k/p-1}T_{r+1+pi}$. The line $\ell$ is part of the
arrangement $I_2(2k)$ containing $\Hf$,
since we assumed that $K$ is stable
by the Coxeter group of this arrangement, we obtain that
the polynomial $f_r-f_{r+1}$ vanishes on $\ell$, and then we deduce
as before that $S_r(K+a)=S_{r+1}(K+a)$.

\end{proof}

If $K$ is a disc centered at $0$ and $k=2p$ then
Proposition~\ref{proposition_k_people}
recovers the following result:

\begin{corollary}
[J., M.\ D., J.\ K., A.\ D. and P.\ M.\ Hirschhorn~\cite{Hirschhorn_times_five}]
Cut a disc with $2p$ lines
through a point in the disc such the lines are equally spaced.
Then $p$ people can have $4$ slices each so that
they each have the same amount of pizza.
\label{corollary_Hirschhorn_Hirschhorn_Hirschhorn_Hirschhorn_Hirschhorn}
\end{corollary}

It is natural to ask if there are systematic generalizations of this kind
of result for higher-dimensional hyperplane arrangements
that give
ways to share a pizza equally between more than two people.
We present two such results in the following two remarks.

\begin{remark}
{\rm
Using an arrangement $\Hf$ of type $F_{4}$ in $\Rrr^4$ we can
divide a pizza evenly among $4$ people.
Let $W$ denote the Coxeter group generated by $\Hf$.
Let $a\in\Rrr^4$ and let $K$ be a measurable set of finite volume
that contains the convex hull of $\{w(a) : w \in W\}$.
Finally, let $L$ be the translation $K+a$.

We first note that we can write the arrangement $\Hf$
as a disjoint union $\Hf_{1} \sqcup \Hf_{2}$, where~$\Hf_{1}$ and~$\Hf_{2}$ are both of type $D_{4}$.
This partition of $\Hf$ corresponds to the long and short roots
in the crystallographic root system $F_{4}$.
Now we can assign three signs
$\vec{s}(T) = (s,s_{1},s_{2})$ to a chamber $T \in \Tf(\Hf)$.
First let $s$ be $(-1)^{T}$ where the sign is computed
with respect to the arrangement~$\Hf$.
For $i=1,2$ let $T^{(i)}$ be the unique chamber
in $\Tf(\Hf_{i})$ that contains $T$.
Let $s_{i}$ be the sign $(-1)^{T^{(i)}}$ with respect to~$\Hf_{i}$.
Observe that $s = s_{1} \cdot s_{2}$ holds since 
the three separation sets satisfy
$S_{\Hf}(T_{0},T)
=
S_{\Hf_{1}}(T_{0}^{(1)},T^{(1)})
\sqcup
S_{\Hf_{2}}(T_{0}^{(2)},T^{(2)})$.
Hence there are only four possible sign patterns for
$\vec{s}(T)$, which are the elements of
$P = \{(1,1,1), (1,-1,-1), (-1,1,-1), (-1,-1,1)\}$.
We use the set $P$ to label the four people sharing the pizza.
For $p \in P$ let $V_{p}$ be the amount of pizza
that person $p$ receives, that is,
$$
V_{p}
=
\sum_{\substack{T \in \Tf(\Hf) \\ \vec{s}(T) = p}} \Vol(T \cap L) .
$$
Since the hyperplane $\Hf$ satisfies Theorem~\ref{theorem_amazing}, we obtain
\begin{align*}
V_{(1,1,1)}
+
V_{(1,-1,-1)}
& =
V_{(-1,1,-1)}
+
V_{(-1,-1,1)}
=
\Vol(L)/2 .
\end{align*}
Similarly, since $\Hf_{1}$ and $\Hf_{2}$ also satisfy Theorem~\ref{theorem_amazing},
the following two equalities hold:
\begin{align*}
V_{(1,1,1)}
+
V_{(-1,1,-1)}
& =
V_{(1,-1,-1)}
+
V_{(-1,-1,1)}
=
\Vol(L)/2 , \\
V_{(1,1,1)}
+
V_{(-1,-1,1)}
& =
V_{(1,-1,-1)}
+
V_{(-1,1,-1)}
=
\Vol(L)/2 .
\end{align*}
Solving this linear equation system yields
\begin{align*}
V_{(1,1,1)}
& =
V_{(1,-1,-1)}
=
V_{(-1,1,-1)}
=
V_{(-1,-1,1)}
=
\Vol(L)/4 .
\end{align*}
}
\label{remark_F_4}
\end{remark}

\begin{remark}
{\rm
Let $n \geq 2$ be even and
let $\Hf$ be an arrangement of type~$B_{n}$ in $\Rrr^n$.
Using the same idea as the previous remark,
we can divide the boundary (crust) of an $n$-dimensional ball
$\Ball(a,R)$, where $\|a\| \leq R$, evenly among four people.
Note that the arrangement $\Hf$ can be written as the
disjoint union of two Coxeter arrangements of types~$D_{n}$ and~$A_{1}^{n}$ for $n \geq 4$
and a disjoint union of two arrangements of type $A_{1}^{2}$ when $n=2$.
Again, we can assign three signs to each chamber, and
we obtain one of the four sign patterns
in the set $P$.
Let~$S_{p}$ denote the sum
\begin{align*}
S_{p}
& =
\sum_{\substack{T \in \Tf(\Hf) \\ \vec{s}(T) = p}}
\Vol_{n-1}(T \cap \Sss(a,R)) .
\end{align*}
By using Theorem~\ref{theorem_surface_volume}
and
reasoning similar to that of Remark~\ref{remark_F_4},
we obtain 
\begin{align*}
S_{(1,1,1)}
& =
S_{(1,-1,-1)}
=
S_{(-1,1,-1)}
=
S_{(-1,-1,1)}
=
\Vol_{n-1}(\Sss(a,R))/4 .
\end{align*}
}
\end{remark}

Finally, let us end with a conjecture about the type~$A$ arrangement.
The $A_{n}$ arrangement lies inside
the $n$-dimensional space
$\{(x_{1},x_{2}, \ldots,x_{n+1}) : x_{1} + x_{2} + \cdots + x_{n+1} = 0\}$
and consists of the $\binom{n+1}{2}$ hyperplanes
$x_{i} = x_{j}$ where $1 \leq i < j \leq n+1$.

\begin{conjecture}
Let $\Hf$ be a hyperplane arrangement of type~$A_{n}$ with
$n \equiv 2,3 \bmod 4$, and let $a\in V$ such that $\|a\|\leq R$.
Then the pizza quantity~$P(\Hf,\Ball(a,R))$ 
is zero if and only if the point~$a$ lies
on one of the hyperplanes of~$\Hf$.
Furthermore, when the point~$a$ belongs 
to the interior of a chamber $T$ of the arrangement,
the sign of the pizza quantity~$P(\Hf,\Ball(a,R))$
is $(-1)^{T}$, that is, the sign of the chamber~$T$.
\end{conjecture}
This conjecture 
is true in dimension~$2$;
see Theorem~1 in~\cite{Mabry_Deiermann}
in the case of $3$ lines.
What can be said about the other
irreducible Coxeter arrangements,
that is, type~$D_{n}$ where $n$ is odd?

\section*{Acknowledgements}

The authors thank Dominik Schmid
for introducing us to the Pizza Theorem.
We also thank Theodore~Ehrenborg and the referee
for their comments on the manuscript.
This work was partially 
supported by the LABEX MILYON (ANR-10-LABX-0070) of Universit\'e 
de Lyon, within the program ``Investissements d'Avenir'' (ANR-11-IDEX-0007)
operated by the French National Research Agency (ANR),
and by Princeton University.
The third author also thanks Princeton University
for hosting a one-week visit in Spring 2020,
and the second author thanks the University of
Kentucky for its hospitality during a one-week visit in Fall 2019.
This work was also partially supported by grants from the
Simons Foundation
(\#429370 to Richard~Ehrenborg
and \#422467 to Margaret~Readdy).

\newcommand{\journal}[6]{{\sc #1,} #2, {\it #3} {\bf #4} (#5), #6.}
\newcommand{\book}[4]{{\sc #1,} #2, {\it #3,} #4.}
\newcommand{\bookf}[5]{{\sc #1,} #2, {\it #3,} {\it #4,} #5.}
\newcommand{\arxiv}[3]{{\sc #1,} #2, {\tt #3}.}
\newcommand{\preprint}[3]{{\sc #1,} #2, preprint {(#3)}.}
\newcommand{\preparation}[2]{{\sc #1,} #2, in preparation.}

\end{document}